\documentclass{amsart}
\usepackage{graphicx}
\usepackage{amsmath,amsthm,amssymb,amscd}
\usepackage[german,frenchb,english]{babel}
\usepackage{hyperref}
\usepackage[all,cmtip]{xy}
\input diagxy
\usepackage[margin=1.2in]{geometry}

\usepackage{enumitem} 

\newcommand{\triv}{\textit{triv}}

\newcommand{\sSet}{\mathsf{sSet}}

\newcommand{\sCat}{\mathsf{sCat}}
\newcommand{\sSetr}{\mathsf{sSet}_0}
\newcommand{\sGrp}{\mathsf{sGrp}}
\DeclareMathOperator{\ob}{ob}
\DeclareMathOperator{\Exh}{Ex^{1,h}}
\DeclareMathOperator{\Ex}{Ex^1}
\DeclareMathOperator{\sd}{sd}

\newcommand{\Ddelta}{\Delta}
\newcommand{\Eb}{\mathbb{E}}

\newcommand{\into}{\hookrightarrow}
\newcommand{\onto}{\twoheadrightarrow}
\DeclareMathOperator{\Ch}{Ch}

\DeclareMathOperator{\Cl}{Cl}

\DeclareMathOperator{\rk}{rk}

\DeclareMathOperator{\PicZR}{Pic^{\mathbb{Z}}_R}

\DeclareMathOperator{\FermZ}{Ferm^{\mathbb{Z}}}

\DeclareMathOperator{\detZ}{det^{\mathbb{Z}}}

\DeclareMathOperator{\E}{\mathcal{E}}

\DeclareMathOperator{\Prok}{\mathsf{Pro}^{\textit{a}}_{\kappa}}

\newcommand{\Tc}{\mathcal{T}}
\newcommand{\Excat}{\mathsf{Cat}_{\textit{ex}}}

\newcommand{\Index}{\mathsf{Index}}

\newcommand{\Dc}{\mathcal{D}}
\newcommand{\Hc}{\mathcal{H}}
\newcommand{\Fred}{Fred}
\newcommand{\elTatec}{\mathsf{Tate}^{\textit{el}}_{\aleph_0}}
\newcommand{\twoTate}{2\text{-}\mathsf{Tate}^{\textit{el}}}
\newcommand{\Tatec}{\mathsf{Tate}_{\aleph_0}}
\newcommand{\Tatea}{\mathsf{Tate}_{\aleph_0}}
\newcommand{\elTatea}{\mathsf{Tate}^{\textit{el}}_{\aleph_0}}
\newcommand{\Indc}{\mathsf{Ind}^{\textit{a}}_{\aleph_0}}

\newcommand{\Tatek}{\mathsf{Tate}_{\kappa}}
\newcommand{\elTatek}{\mathsf{Tate}^{\textit{el}}_{\kappa}}
\newcommand{\elTatekp}{\mathsf{Tate}^{\textit{el}}_{\kappa'}}

\newcommand{\Spaces}{\mathsf{Spaces}}

\DeclareMathOperator{\Vect}{Vect}

\DeclareMathOperator{\Perf}{Perf}
\DeclareMathOperator{\coker}{coker}

\DeclareMathOperator{\ic}{ic}

\DeclareMathOperator{\Indk}{\mathsf{Ind}^{\textit{a}}_{\kappa}}
\DeclareMathOperator*{\colim}{\varinjlim}
\DeclareMathOperator*{\hocolim}{hocolim}

\DeclareMathOperator{\Mod}{\textsf{Mod}}
\DeclareMathOperator{\Hom}{\textsf{Hom}}

\DeclareMathOperator{\Pic}{Pic}

\newcommand{\C}{\mathsf{C}}
\newcommand{\Det}{Det}
\newcommand{\Gr}{Gr}

\DeclareMathOperator{\grp}{\times}

\DeclareMathOperator{\Set}{Set}

\DeclareMathOperator{\Pro}{\mathsf{Pro}^{\textit{a}}}

\DeclareMathOperator{\Ind}{\mathsf{Ind}^{\textit{a}}}
\DeclareMathOperator{\Fun}{Fun}

\DeclareMathOperator{\op}{op}

\DeclareMathOperator{\Cat}{Cat}

\DeclareMathOperator{\Kb}{\mathbb{K}}

\DeclareMathOperator{\Aut}{Aut}

\DeclareMathOperator{\GL}{GL}

\newcommand{\Zb}{\mathbb{Z}}
\DeclareMathOperator{\D}{\mathsf{D}}

\DeclareMathOperator{\Kk}{\mathcal{K}}

\DeclareMathOperator{\Gg}{\mathcal{G}}

\DeclareMathOperator{\Pc}{\mathcal{P}}

\DeclareMathOperator{\Spec}{Spec}

\DeclareMathOperator{\Tate}{\mathsf{Tate}}
\DeclareMathOperator{\elTate}{\mathsf{Tate}^{\textit{el}}}

\begin{document}
\newtheorem{definition}{Definition}[section]
\newtheorem{theorem}[definition]{Theorem}
\newtheorem{proposition}[definition]{Proposition}
\newtheorem{corollary}[definition]{Corollary}
\newtheorem{conj}[definition]{Conjecture}
\newtheorem{lemma}[definition]{Lemma}

\newtheorem{cl}[definition]{Claim}
\newtheorem{example}[definition]{Example}
\newtheorem{problem}[definition]{Problem}
\newtheorem*{nproblem}{Problem}
\newtheorem{claim}[definition]{Claim}
\newtheorem{ass}[definition]{Assumption}
\newtheorem{warning}[definition]{Warning}
\newtheorem{porism}[definition]{Porism}
\newtheorem{notation}[definition]{Notation}
\theoremstyle{remark}
\newtheorem{rmk}[definition]{Remark}
\theoremstyle{plain}

\title{The Index Map in Algebraic $K$-Theory}
\author{Oliver Braunling,\quad Michael Groechenig,\quad Jesse Wolfson}

\address{Department of Mathematics, Universit\"{a}t Freiburg}
\email{oliver.braeunling@math.uni-freiburg.de}
\address{Department of Mathematics, Imperial College London}
\email{m.groechenig@imperial.ac.uk}
\address{Department of Mathematics, University of California - Irvine}
\email{wolfson@uci.edu}

\begin{abstract}
For a ring $R$, {we give a new construction of the universal $K_R$-torsor $\mathcal{T}_{K}\rightarrow K_{\Tate(R)}$ constructed by Sho Saito}, and its analogue for general idempotent complete exact categories. We study the classifying map of this torsor in detail, showing that it is an equivalence, relate it to a boundary map in a $K$-theory localization sequence, construct an explicit simplicial model, and link it to the index theory of Fredholm operators. The torsor is also related to canonical central extensions of loop groups. Just like classical loop group theory has features of $K$-theory (e.g. determinant bundles, the tame symbol cocycle for Kac-Moody extensions), the $K$-theory torsor relates higher loop groups with higher $K$-theory. We compare the $K$-theory torsor to previously studied dimension and determinant torsors.
\end{abstract}

\thanks{O.B.\ was supported by DFG SFB/TR 45 ``Periods, moduli spaces and arithmetic of algebraic varieties'' and Alexander von Humboldt Foundation. M.G.\ was supported by EPRSC Grant No.\ EP/G06170X/1. J.W.\ was partially supported by an NSF Graduate Research Fellowship under Grant No.\ DGE-0824162, by an NSF Research Training Group in the Mathematical Sciences under Grant No.\ DMS-0636646, and by an NSF Post-doctoral Research Fellowship under Grant No.\ DMS-1400349. He was a guest of K. Saito at IPMU while this paper was being completed. Our research was supported in part by NSF Grant No.\ DMS-1303100 and EPSRC Mathematics Platform grant EP/I019111/1.}


\maketitle

\section{Introduction}

 This paper applies algebraic $K$-theory to solve a problem posed independently by Kapranov \cite{Kap:Bryl} and Drinfeld \cite{MR2181808}. To state the problem, recall that a central feature in the theory of loop groups is the Kac-Moody central extension. Over an equicharacteristic local field, Kapranov observed that this extension arises from a natural ``determinantal'' torsor.  Variations of this torsor appear in a wide array of geometric settings, including Kontsevich's construction of motivic integration \cite{motivic}, Kapranov and Vasserot's  construction of the chiral de Rham complex \cite{MR2102701}, Beilinson--Bloch--Esnault's work on de Rham $\varepsilon$-factors \cite{MR1988970}, Drinfeld's approach to the Uhlenbeck compactification of the moduli of vector bundles on an algebraic surface \cite{MR2181808} {\cite{MR3431633}}, and Chinburg--Pappas--Taylor's arithmetic higher Riemann-Roch theorem \cite{Chinburg:2012fk}.

 {
The theory of Tate objects in exact categories provides a natural framework to study all of these constructions in a unified manner. We refer the reader to Section \ref{sec:prelim} for a more detailed review. The following result of Saito \cite[Theorem 1.2]{Saito:2012uq} on the algebraic $K$-theory of Tate objects casts a new light on the aforementioned constructions.
\begin{theorem}[Saito]\label{saito}
For every idempotent complete exact category $\C$, we have an equivalence of non-connective $K$-theory spectra $\Omega\Kb_{\elTate(\C)} \simeq \Kb_{\C}$.
\end{theorem}

Even before Saito's work appeared, Kapranov and Drinfeld (the latter following Beilinson) observed that the structure of determinantal torsors was reminiscent of $K$-theory. This motivated the following.

}

    \begin{nproblem}\cite[Problem 5.5.3]{MR2181808}
        ``The notion of determinantal Torsor is very useful, and its rigorous interpretation in the standard homotopy-theoretic language of algebraic $K$-theory would be helpful.''\linebreak
  {Specifically, one has to
\begin{enumerate}[label=(\Alph*)]
\item Clarify what the term ``torsor'' should mean when talking about a generalized $K$-theory torsor.
\item Relate the torsor to the delooping (or the ``Calkin category'' in the context \textit{loc. cit.})
\item Show that the $K$-theory torsor indeed truncates to the graded determinant torsor.
\end{enumerate}}
    \end{nproblem}
 {Due to the higher homotopical nature of the $K$-theory space it is not immediately obvious what is meant by a $K$-theory torsor. While the notion of torsors over such objects has been implicit in the homotopy theory literature for some time, it has received renewed attention in the framework of $\infty$-topoi \cite{nikolaus2015principal}.  A short summary of this approach is given in \cite[Section 2]{Saito:2014fk}. Using this framework, Saito answered the first two parts of Drinfeld's problem as a consequence of his delooping theorem.  For the present article, we use the following minimalistic approach to torsors: we have an equivalence of $\infty$-categories (see Theorem 2.19 in \cite{nikolaus2015principal} or Lemma 7.2.2.1 in \cite{Lurie:bh})
   \[
  B\colon  \{\text{group objects in }\Spaces\} \simeq \{\text{pointed and connected spaces} \}.
   \]
  The functor $B$ will also be called the \emph{classifying space functor}. For a group object $G$ in the $\infty$-category of spaces, specifying a $G$-\emph{torsor} on a space $X$, say
\[
\mathcal{T}\to X,
\]
can be taken to mean to provide a map $X \rightarrow BG$, which is then called \emph{classifying map}; see Theorem 3.17 in \cite{nikolaus2015principal}. We may therefore always work in terms of classifying maps.

  Once this foundational question has been settled, the problem is divided into two tasks. First of all such a $K$-theory torsor needs to be defined and its basic properties proven, and secondly compatibility with the determinantal torsor needs to be established. The first part is addressed by Saito's papers \cite{Saito:2012uq} and \cite{Saito:2014fk}. He \textit{uses} the delooping to construct the torsor, so in his approach the sub-problem (B) is tautologically solved.

In this paper we develop a different point of view which yields a new construction of Saito's torsor. Firstly, we give an explicit simplicial construction of the classifying map, and then prove (B) in our new setting, which is non-trivial because of the very different construction. We then show agreement (up to sign) with Saito's torsor, and finally settle (C), the compatibility with the determinantal torsor.} Specialized to Drinfeld's setting, our main result is as follows.
    \begin{theorem}[{ Corollaries \ref{cor:index} \& \ref{cor:enhanced} \& Proposition \ref{prop:dettor}}]\label{thm:drn}
        Let $R$ be a ring, let $\Tate(R)$ be the category of Tate $R$-modules,
        Then there exists a natural $K_R$-torsor \begin{equation} \label{intro:torsor:ringver}
            \Tc_K\to \Tate(R)^\times \tag{$\star$}
        \end{equation}
        such that the following properties hold:
        \begin{enumerate}
            \item {\bf Multiplicativity}: { A short exact sequence $V_0\into V_1\onto V_2$ of Tate $R$-modules induces an equivalence of torsors
                \begin{equation*}
                    \Tc_K|_{V_0}\otimes_{K_R} \Tc_K|_{V_2}\to^\simeq \Tc_K|_{V_1},
                \end{equation*}
                and}
            \item {\bf Determinant}: The fiberwise 1-truncation of $\Tc_K\to \Tate(R)^\times$ is the (graded) determinantal torsor
                \begin{equation*}
                    \Det\to \Tate(R)^\times.
                \end{equation*}
        \end{enumerate}
    \end{theorem}

    \begin{corollary}
        Let $V$ be a Tate $R$-module.
   \begin{enumerate}
   \item The restriction of the $K_R$-torsor to $B\Aut(V)$ gives the {classifying map of a central extension of $\Aut(V)$ by the infinite loop space $\Omega K_R$.
   \item Composing with the map to the 1-truncation ${\tau}_{\leq 1}K_R$, this yields the Kac-Moody extension of $\Aut(V)$ by $\mathbb{G}_m=\Omega {\tau}_{\leq 1}K_R$}.
  \end{enumerate}
    \end{corollary}
{ Assertion (1) is tautological, since one defines a central extension of a group $G$ by a spectrum $\mathbb{E}$, in terms of its classifying map $G \to \Omega^{\infty} \Sigma^2\Eb$. Assertion (2) follows from Proposition \ref{prop:dettor}.}

 {This extension has also been intensely studied by M. Sato and his school, and the cocycle giving it is frequently referred to as the \emph{Japanese} cocycle.}
    We in fact prove this theorem for Tate objects in an arbitrary idempotent complete exact category $\C$. { In this generality, we obtain a $K_{\C}$-torsor
\[
\Tc_K\to \elTate(\C)^\times \text{,}
\]
taking the place of $\eqref{intro:torsor:ringver}$. We leave the details to the main body of the paper. Saito's Theorem \ref{saito} shows an abstract equivalence of non-connective $K$-theory spaces $\Omega\Kb_{\elTate(\C)} \simeq \Kb_{\C}$ and therefore indicates that a much stronger version of the multiplicativity assertion holds (as conjectured by Kapranov).}
    \begin{theorem}[{Corollaries \ref{cor:enhanced} \& \ref{cor:minuseins}}]\label{thm:torsor}
        Let $\C$ be an idempotent complete exact category.  The $K_{\C}$-torsor $\Tc_K\to \elTate(\C)^\times$, on the classifying space of elementary Tate objects in $\C$, extends to a $K_{\C}$-torsor on the $K$-theory space $K_{\elTate(\C)}$. The classifying map of this extended torsor,
        \begin{equation*}
            \Index\colon K_{\elTate(\C)}\to B K_{\C} \text{,}
        \end{equation*}
        { agrees with the map induced by Saito's $\Kb_{\elTate(\C)} \to \Sigma \Kb_{\C}$ up to sign, and hence}
        is an equivalence of infinite loop spaces.
    \end{theorem}

    \begin{rmk}
        To explain the name of this map, note that Tate objects can be understood as an abstraction of formal loop spaces, and the index map of the theorem then arises from an analogue of the construction of the index bundle of families of Toeplitz operators.
    \end{rmk}

    Our last main result establishes a close relation between the index map and boundary maps in $K$-theory.  To state the most important example, let $R$ be a ring, and denote by $P_f(R((t)))$ the exact category of finitely generated projective $R((t))$-modules. Let
    \begin{equation*}
        T\colon P_f(R((t)))\to\Tate(R)
    \end{equation*}
    be the canonical functor which sends the rank one free module $R((t))$ to the Tate $R$-module $R((t))$. {Combining the functoriality of boundary maps with Theorem \ref{thm:loc} one immediately obtains the following theorem.}

    \begin{theorem}[{\bf Boundary}]\label{thm:bdry}
        There exists a canonical homotopy commuting triangle (i.e. commuting triangle in the $\infty$-category of spaces)
        \begin{equation*}
            \xymatrix{
                K_{R((t))} \ar[d]_T \ar[rr]^{-\partial} && \tau_{\ge 0}\Sigma\Kb_R \\
                K_{\Tate(R)} \ar[urr]_{\Index}
            }.
        \end{equation*}
        where $\tau_{\ge 0}\Sigma\Kb_R$ denotes the (infinite looping) of the connective cover of the suspension of the non-connective $K$-theory spectrum of $R$.
    \end{theorem}

%

{For the rest of the introduction, let us list further results, which have not been part of Drinfeld's problem:}\medskip

    \subsection*{Operator Theory}
    We view the above as theorems primarily about the operator theory of Tate objects, rather than as theorems about algebraic $K$-theory. As such, for an elementary Tate object $V$, we pay particular attention to the composite
    \begin{equation*}
        B\Aut(V)\to K_{\elTate(\C)}\to^{\Index} BK_{\C}.
    \end{equation*}
    By looping this map, we obtain an $A_\infty$-map
    \begin{equation*}
        \Aut(V)\to K_{\C}.
    \end{equation*}
    In Section \ref{sub:indexaut} and in a sequel \cite{BGW:Segal}, we give two complementary constructions of the $A_\infty$-structure of this map, one combinatorial and suited to non-homotopical settings such as the study of $\Aut_{\Tate(R)}(R((t)))$, and one $\infty$-categorical and suited to homological settings such as the study of Tate objects in stable $\infty$-categories (cf. \cite{Hennion}). These constructions encode computations of the index of automorphisms of Tate objects, and of higher torsion invariants of tuples of such. In \cite{BGW:14}, we apply this to establish the compatibility of our $K$-theoretic definition of higher Contou-Carr\`ere symbols with the existing definitions in dimensions 1 and 2.

    \subsubsection*{{\bf Relation to Index of Fredholm Operators}}
    In Section \ref{sec:aj}, we explain how the map $\Index$ provides a precise analogue of the classical index map
    \begin{equation*}
        \Fred(\Hc)\to K^{top}_{\mathbb{C}}
    \end{equation*}
    from the space of Fredholm operators on complex Hilbert space to the classifying space of topological $K$-theory, cf. J\"anich \cite{Jan:65} and Atiyah \cite{Ati:67}. To stress this analogy, let us state this classical result back to back to our version for Tate objects:
    \begin{theorem}\label{thm:main_aj}\mbox{}
    \begin{enumerate}
        \item Let $(\Hc,\Hc^+)$ be a polarized, separable complex Hilbert space. The space of Fredholm operators $\Fred(\Hc^+)$ is equivalent to $$K^{top}_{\mathbb{C}}\cong\Omega\left(B\GL_{\mathrm{res}}(\Hc,\Hc^+)\right)^+,$$ where the outer superscript $+$ refers to Quillen's \emph{$+$-construction}.
        \item Let $R$ be a ring.  The $R$-module $R((t))$ admits a canonical structure of an elementary Tate object in the category of finitely generated projective $R$-modules, and we have an equivalence $$K_{R} \cong \Omega\left(B\Aut_{\Tate(R)}\,(\,R((t))\,)\right)^+.$$
        \end{enumerate}
    \end{theorem}

    \subsection*{Determinant Torsors and Central Extensions}
    In the spirit of Drinfeld's problem, the $K$-theory torsor of Theorem \ref{thm:torsor} unifies the constructions which have arisen in the literature (e.g. \cite{MR0619827}, \cite{Kapranov:fk}, \cite{MR2656941}, \cite{MR1988970}, \cite{MR2972546}), and extends them from one and two-fold loop groups, to general linear groups over arbitrary higher local fields. Further, the present treatment allows one to bring algebraic $K$-theory to bear on the study of these objects. We discuss these connections in detail in Section \ref{sub:torsor}. In particular, we show the following {(see Subsections \ref{sub:dimtor}-\ref{sec:2ger})}.
    \begin{theorem}\mbox{}
        Let $R$ be a commutative ring.
        \begin{enumerate}
            \item The 1-truncation of the index map
                \begin{equation*}
                    B(\Tate(R)^\times)\to B\Zb
                \end{equation*}
                classifies the dimension torsor on the category of Tate $R$-modules, cf. \cite[Section 3.5]{MR2181808}.
            \item The 2-truncation of the index map
                \begin{equation*}
                    B(\Tate(R)^\times)\to B\Pic^{\Zb}_R
                \end{equation*}
                classifies the graded determinantal torsor on the category of Tate $R$-modules, cf. \cite[Section 5.3]{MR2181808}.
            \item Denote by ${2\text{-}\elTate(R)}$ the category of elementary Tate objects in $\Tate(R)$. The 3-truncation of the index map
                \begin{equation*}
                    B(2\text{-}\elTate(R)^\times)\to B^2\Pic^{\Zb}_R
                \end{equation*}
                classifies the torsors of ``graded gerbal theories'' on the category of $2\text{-}\elTate(R)$. For $R$ a field, this coincides with the 2-gerbe constructed by Arkhipov and Kremnitzer \cite{MR2656941}.
                \item { The morphism $- \Index$ corresponds to the torsor constructed by Saito in \cite{Saito:2014fk}.}
        \end{enumerate}
    \end{theorem}
    A primary motivation for studying these torsors is that they allow one to define higher Kac-Moody extensions for iterated loop groups, or, more generally, for reductive groups over higher local fields.  A key feature of these extensions, which we pursue in depth in a sequel \cite{BGW:14}, is that they deform over families of higher local fields.  In \emph{loc. cit.}, we use the tools of this paper to relate the 1 and 2-dimensional Contou-Carr\`ere symbols to these $K$-theoretic extensions.  We also use $K$-theory to \emph{define} the higher Contou-Carr\`ere symbol over a family of $n$-dimensional local fields, and we then use standard techniques of $K$-theory to give a short, conceptually simple proof of reciprocity laws in all dimensions.

    In the present paper, we restrict our attention to the torsors which have previously appeared in the literature, recovering those discussed above, and also constructing a torsor sketched by Frenkel and Zhu \cite{MR2972546} (see Section \ref{sec:FZ} for more details). As Frenkel and Zhu emphasize, this torsor should be the starting point of a ``geometric representation theory'' of higher extensions of double loop groups.  To date, only tentative steps have been taken toward understanding the representation theory of iterated loop groups, e.g. \cite{MR2100671}, \cite{MR2192062}, \cite{MR2214248}, \cite{MR2656941}, \cite{MR2972546}, and \cite{Safranov}. We hope that the methods of this paper should enable future progress in this direction.

\subsection*{Outline of the Paper}
    This paper is formulated in the language of Tate objects in exact categories and algebraic $K$-theory. In Section \ref{sub:exact} we recall the relevant background on exact categories and Tate objects.  A primary reference for exact categories is B\"uhler's excellent survey \cite{MR2606234}, while our earlier article \cite{Braunling:2014fk} contains the necessary background on Tate objects. We next proceed, in Section \ref{sub:k}, to recall key facts about the algebraic $K$-theory of exact categories, following Waldhausen's framework \cite{MR0802796}.  In particular, we recall in some detail Waldhausen's construction of the cofiber sequence associated to a map of exact categories; this will form the basis for our comparison of the index map and the boundary map in $K$-theory. Experts in $K$-theory should feel free to skip this section, though we remark that the material on ``left path spaces'' (i.e. Definition \ref{def:relwald} -- Proposition \ref{prop:Slprops}) is not in Waldhausen and may prove useful beyond the present work. With this in hand, we provide, in Section \ref{subsub:boundary}, a precise treatment of boundary maps in the $K$-theory localization sequence (cf. Proposition \ref{prop:waldbound}). While this description is strongly suggested by Waldhausen, it is not established in \emph{loc. cit.} and we were unable to find an independent reference. In the process of giving this description, we also pin down a sign ambiguity in the boundary map in the $K$-theory localization sequence (cf. Proposition \ref{prop:Schlichting}).

    In Section \ref{index}, we construct the index map and establish its key properties as detailed in Theorems \ref{thm:drn}, \ref{thm:torsor} and \ref{thm:bdry}. En route, we construct a simplicial map whose geometric realization is homotopy equivalent to the index map; this gives a ``Kan-ian'' construction of an $A_\infty$-homomorphism $\Aut(V)\to K_{\C}$ and encodes index and torsion invariants of automorphisms of elementary Tate objects. In a sequel \cite{BGW:Segal}, we revisit the question of the $A_\infty$-structure, and give a ``Segal-ian'' construction of this map, which has the benefit of working mutatis mutandis for Tate objects in stable $\infty$-categories as considered by Hennion.

    In Section \ref{sec:aj} we compare our constructions with classical ones. We compare the index map for Tate objects to the index theory of Fredholm operators.  We then construct the $K$-theory torsor, describe its sections, and study its truncations in connection with previous work.

\subsection*{Acknowledgements}
We thank Y. Kremnitzer for introducing us first to these questions and then to each other.  We thank E. Getzler, M. Kapranov, R. Nest and B. Tsygan for helpful conversations. We would like to thank T. Hausel for supporting a visit of the first and the third author to EPF Lausanne, where part of this work was carried out. J.W. also thanks K. Saito and IPMU for the pleasant working conditions while this article was being completed. We would like to thank A. Beilinson and V. Drinfeld for supporting a visit of the first and second author to the University of Chicago, where this paper was completed. We thank I. Zakharevich for pointing out a mistake in an earlier version of the paper. {We thank the anonymous referee for greatly improving the presentation of the material.}

\section{Preliminaries}\label{prelim}\label{sec:prelim}
\subsection{Exact Categories}\label{sub:exact}
The first paragraph of this subsection is devoted to the definition of exact categories, and related notions. In Paragraph \ref{subsub:tate} we recall the basics of Ind, Pro, and Tate objects, which form the main players of the present text.

\subsubsection{Definitions}\label{subsub:defi}
For the convenience of the reader, we recall the definition of exact categories below. More details can be found in B\"{u}hler's excellent survey \cite{MR2606234}.

\begin{definition}\label{defi:exact}
We denote by $\C$ an \emph{additive} category.
\begin{itemize}
    \item[(a)] A \emph{kernel-cokernel pair} is given by maps
    \begin{equation*}
        X\hookrightarrow Y\twoheadrightarrow Z
    \end{equation*}
    with $X\hookrightarrow Y$ being the kernel of $Y\twoheadrightarrow Z$, and $Y\twoheadrightarrow Z$ being the cokernel of $X\hookrightarrow Y$.

    \item[(b)] The structure of an \emph{exact category} on $\C$ is given by a class $\E\C$ of kernel-cokernel pairs. A map $X\hookrightarrow Y$, serving as the kernel in a kernel-cokernel pair in $\E\C$, is called an \emph{admissible monic}.  Similarly, a map $Y \twoheadrightarrow Z$, which serves as a cokernel in a kernel-cokernel pair of $\E\C$, is called an \emph{admissible epic}. The following axioms have to be satisfied:
    \begin{enumerate}
        \item identity morphisms are admissible monics and admissible epics,
        \item the composition of two admissible monics is an admissible monic, similarly for admissible epics,
        \item pushouts of admissible monics along arbitrary maps exist, and are again admissible monics; similarly, pullbacks of admissible epics along arbitrary maps exist and are again admissible epics.
    \end{enumerate}
    \item[(c)] An additive functor $F\colon \C \to \D$ is called \emph{exact}, if it maps $\E\C$ to $\E\D$. If $F\colon\C\to\D$ is fully faithful and exact, we call it \emph{fully exact}, if every kernel-cokernel pair $F(X)\hookrightarrow F(Y)\twoheadrightarrow F(Z)$ in $\E\D$ stems from a kernel-cokernel pair in $\C$.
    \item[(d)] We note by $\Excat$ the 2-category of exact categories, exact functors, and natural transformations.
\end{itemize}
\end{definition}

We will refer to the kernel-cokernel pairs in $\E\C$ as \emph{short exact sequences} or \emph{extensions} in the exact category $\C$.

\begin{example}
For a ring $R$ we denote by $P_f(R)$ the additive category of finitely generated projective modules. Considering kernel-cokernel pairs obtained from short exact sequences of $R$-modules, with all constituents being finitely generated projective, we obtain a natural exact structure on this category.
\end{example}

The next definition recalls the notion of being \emph{idempotent complete}. This notion is reminiscent of the properties of linear projectors, familiar from linear algebra.

\begin{definition}
We say that an additive category $\C$ is \emph{idempotent complete}, if every idempotent splits, i.e. for every morphism $p\colon X\to X$, satisfying $p^2=p$, we have an isomorphism $X\cong Y\oplus Z$ taking $p$ to the idempotent $0 \oplus 1_Z$.
\end{definition}

Every exact category can be embedded into an idempotent complete exact category, in an essentially unique way.

\begin{proposition}(cf. \cite[Proposition 6.10]{MR2606234})
For every exact category $\C$ there exists a fully exact embedding $\C\hookrightarrow\C^{\ic}$ into an idempotent complete exact category, which is $2$-universal with respect to this property.
\end{proposition}

Following Schlichting \cite[Def. 1.3 \& 1.5]{MR2206639}, we recall the notion of left s-filtering subcategories of exact categories. The perk of left s-filtering inclusions is that a corresponding quotient can be formed in the $2$-category of exact categories. The definition given below is an equivalent, but simpler version, which was communicated to us by B\"uhler (see \cite[App. A]{Braunling:2014fk} for a reproduction of B\"uhler's argument, which compares the definition below with Schlichting's).

\begin{definition}\label{defi:sfilt}
Let $\C \hookrightarrow \D$ be a fully faithful exact functor between exact categories.
\begin{itemize}
\item[(a)] The inclusion is \emph{left special}, if for every object $Z \in \C$, and every admissible epic $G \twoheadrightarrow Z$ in $\D$, we have a commutative diagram with exact rows in $\D$
\[
\xymatrix{
X \ar@{^{ (}->}[r] \ar[d] & Y \ar@{->>}[r] \ar[d] & Z \ar[d]^{1_Z} \\
F \ar@{^{ (}->}[r] & G \ar@{->>}[r] & Z,
}
\]
with the top row being an extension in $\C$. The inclusion is called \emph{right special}, if $\C^{\op} \hookrightarrow \D^{\op}$ is left special.
\item[(b)] The inclusion is \emph{left filtering}, if every morphism $Y \to F$ in $\D$, with $Y \in \C$, factors through an admissible monic $Z \hookrightarrow F$ with $Z\in\C$:
\[
\xymatrix{
Y \ar[r] \ar[rd] & F \\
& Z. \ar@{^{ (}->}[u]
}
\]
It is called \emph{right filtering} if $\C^{\op} \hookrightarrow \D^{\op}$ is left filtering.
\item[(c)] It is \emph{left s-filtering}, if it is both left special and left filtering. It is called \emph{right s-filtering} if $\C^{\op} \hookrightarrow \D^{\op}$ is left s-filtering.
\end{itemize}
\end{definition}

With this definition in hand, we are able to recall Schlichting's exact quotient category $\D/\C$.

\begin{definition}\label{defi:quotient}
Consider a left s-filtering inclusion of exact categories $\C \hookrightarrow \D$ (see Definition \ref{defi:sfilt}).
\begin{itemize}
\item[(a)] Let $\Sigma_e$ be the class of morphisms in $\D$, which are admissible epics with kernel in $\C$. The left s-filtering condition guarantees that $\Sigma_e$ satisfies a calculus of left fractions (see \cite[Lemma 1.13]{MR2079996}).\footnote{Conversely, as we learned from private correspondence with B\"uhler, if $\Sigma_e$ satisfies a calculus of left fractions, then $\C$ is left filtering and closed under extensions in $\D$.}
\item[(b)] Following \cite[Def. 1.14]{MR2079996}, the localization $\D[\Sigma_e^{-1}]$ will be denoted by $\D/\C$. It inherits an exact structure, by considering the images of short exact sequences in $\D$ in the additive category $\D/\C$ (\cite[Prop. 1.16]{MR2079996}).
\end{itemize}
\end{definition}

\subsubsection{Admissible Ind, Pro and Tate Objects}\label{subsub:tate}
This paragraph provides an informal introduction to Tate objects in exact categories.  Full details are given in \cite{Braunling:2014fk}. Let $\C$ be an exact category.  In \emph{loc. cit.}, the authors defined, for an infinite cardinal $\kappa$,
\begin{enumerate}
    \item the exact category $\Indk(\C)$ of \emph{admissible Ind-objects} in $\C$ (\cite[Def. 3.3]{Braunling:2014fk}),
    \item the exact category $\Prok(\C)$ of \emph{admissible Pro-objects} in $\C$ (\cite[Def. 4.1]{Braunling:2014fk}),
    \item the exact category $\elTatek(\C)$ of \emph{elementary Tate objects} in $\C$ (\cite[Def. 5.1]{Braunling:2014fk}), and
    \item the exact category $\Tatek(\C)$ of \emph{Tate objects} in $\C$ as the idempotent completion of $\elTatek(\C)$ (\cite[Def. 5.26]{Braunling:2014fk}).
\end{enumerate}

The study of these constructions goes back at least to Lefschetz \cite[Ch. II.25]{MR0007093}, Artin--Mazur \cite{MR0245577}, Kato \cite{MR1804933}, Beilinson \cite{MR923134}, and Kapranov \cite{MR1800352}. For $\kappa=\aleph_0$, a recent treatment has also been given by Previdi in \cite{MR2872533}.  For a ring $R$, Drinfeld \cite{MR2181808} has studied a related notion of Tate $R$-module.  The category of countably generated Tate $R$-modules in Drinfeld's sense is equivalent to the category $\Tate_{\aleph_0}(P_f(R))$ (\cite[Thm. 5.30]{Braunling:2014fk}).  In general, Drinfeld's category of Tate modules is a fully exact sub-category of $\Tate(P_f(R))$.  For uncountable cardinalities, the authors provide a geometric interpretation in terms of \emph{flat Mittag-Leffler modules} of the category $\Tate(P_f(R))$ in \cite{Braunling:2014fk}, based on work by \v S\v tov\'\i\v cek and Trlifaj (\cite[App. B]{Braunling:2014fk}).

For the exact category $\Vect_f(k)$ of finitely generated vector spaces over a discrete field $k$, the category $\Indk(\Vect_f(k))$ is equivalent to the category of discrete vector spaces generated by a basis of cardinality at most $\kappa$.  A guiding example is the vector space $k[x]\in\Indc(\Vect_f(k))$.  The category $\Prok(\Vect_f(k))$ is equivalent to the category of topological duals of discrete vector spaces of cardinality at most $\kappa$.  The topological vector space $k[[t]]\cong(k[x])^{\vee}$ is an important example, where the topology on $k[[t]]$ is the $t$-adic topology, i.e. the finest linear topology such that the sequence $\{t^n\}_{n\in\mathbb{N}}$ converges to $0$.  The category $\elTatek(\Vect_f(k))$ is equivalent to the category of topological vector spaces of the form $V\oplus W^{\vee}$ where $V$ and $W$ are discrete vector spaces of cardinality at most $\kappa$. By definition, this is Lefschetz's category of \emph{locally linearly compact} vector spaces \cite[Ch. II.25]{MR0007093}. The archetypical example is the topological vector space $k((t))\cong k[[t]]\oplus t^{-1}k[t^{-1}]\in\elTatek(\Vect_f(k))$.

The categories of admissible Ind-objects, admissible Pro-objects, and elementary Tate objects are related by a commuting square of fully exact embeddings
\begin{equation}\label{eqn:fundamental_diagram}
\xymatrix{
\C \ar@{^{(}->}[r] \ar@{^{(}->}[d] & \Indk(\C) \ar@{^{(}->}[d] \\
\Prok(\C) \ar@{^{(}->}[r] & \elTatek(\C)
}
\end{equation}

The inclusion functors in this diagram are well-behaved with respect to taking quotients (see Definition \ref{defi:sfilt}).

\begin{proposition}\label{prop:filtering}(\cite[Prop. 3.10, 5.8, 5.10 \& 5.32]{Braunling:2014fk})
Let $\C$ be an exact category. The inclusions $\C \hookrightarrow \Indk(\C)$ and $\Prok(\C) \hookrightarrow \elTatek(\C)$ are left s-filtering. The inclusion $\Indk(\C) \hookrightarrow \elTatek(\C)$ is right filtering. The quotient exact categories $\Indk(\C)/\C$ and $\elTatek(\C)/\Prok(\C)$ are equivalent with respect to the map induced by the inclusion $$(\Indk(\C),\C) \hookrightarrow (\elTatek(\C),\Prok(\C))$$
of pairs of exact categories.
\end{proposition}

\begin{rmk}
    If $\C$ is idempotent complete, then as a consequence of \cite[Theorem 6.7]{Braunling:2014fk}, one can also show that the inclusion $\Indk(\C) \hookrightarrow \elTatek(\C)$ is right special.
\end{rmk}

Following Sato--Sato \cite{MR730247}, we consider the set of all \emph{lattices} in an elementary Tate object. The archetypical example of such is the inclusion
$$k[[t]] \subset k((t)).$$
We observe that $k[[t]]$ is a Pro-object, and the quotient $\bigoplus_{n \geq 1}k\langle t^{-n}\rangle$ is an Ind-object. This is the defining quality of \emph{lattices}.

\begin{definition}\label{defi:sato}
    Let $V$ be an elementary Tate object in $\C$.
    \begin{enumerate}
        \item A \emph{lattice} $L\hookrightarrow V$ of an elementary Tate object is an admissible sub-object, with $L\in\Prok(\C)\subset\elTatek(\C)$ and the cokernel $V/L\in\Indk(\C)\subset\elTatek(\C)$.
        \item The \emph{Sato Grassmannian} $\Gr(V)$ is the partially ordered set of lattices in $V$, where $L_0\leq L_1$ if there exists a commuting triangle of admissible monics
            \begin{equation*}
                \xymatrix{
                L_0 \ar@{^{(}->}[r] \ar@{^{(}->}[rd] & L_1 \ar@{^{(}->}[d] \\
                & V
                }
            \end{equation*}
    \end{enumerate}
\end{definition}

Lattices and the Sato Grassmannian play a key role in our study of Tate objects. Assertion (c) in the theorem below, is viewed by the authors as the main result of \cite{Braunling:2014fk}.
\begin{theorem}\label{thm:sato_filtered}(\cite[Prop. 6.6, Thm. 6.7]{Braunling:2014fk})
    Let $\C$ be an exact category.
    \begin{enumerate}
        \item[(a)] Every elementary Tate object in $\C$ has a lattice.
        \item[(b)] The quotient of a lattice by a sub-lattice is an object of $\C$.
        \item[(c)] If $\C$ is idempotent complete, and $L_0\hookrightarrow V$ and $L_1\hookrightarrow V$ are two lattices in an elementary Tate object $V$, then there exists a lattice $N\hookrightarrow V$ with $L_0,L_1\le N$ in $\Gr(V)$. Similarly, $L_0$ and $L_1$ have a common sub-lattice $M \subset L_0,L_1$.
    \end{enumerate}
\end{theorem}

The following convention helps us to avoid awkward notation.

\begin{rmk}\label{rmk:notation}
From now on we consider the infinite cardinal $\kappa$ as fixed, and omit the cardinality bound from the notation, i.e. for an exact category $\C$ we denote by $\Ind(\C)$, $\Pro(\C)$, and $\elTate(\C)$ the corresponding exact categories.
\end{rmk}

The justification for this omission is that \emph{algebraic $K$-theory} is not sensitive to a change of the cardinality $\kappa$ (cf. Corollary \ref{cor:tatek=k'}).

\subsection{Algebraic $K$-Theory}\label{sub:k}
The results of this article are formulated in the language of algebraic $K$-theory.  Primary references include Quillen \cite{MR0338129}, Waldhausen \cite{MR0802796}, and  Schlichting {\cite{MR2079996, MR2206639}. In addition to recalling necessary material and fixing notation, the primary purpose of this section is to prove that a map which arises in Waldhausen \cite{MR0802796} is, up to sign, the boundary map in the $K$-theory localization sequence (cf. Propositions \ref{prop:waldbound} and \ref{prop:Schlichting}). Of course, the choice of sign depends on choosing an orientation of the path space of a based space $X$. We follow convention (and Waldhausen) and define $PX$ to be space of paths \emph{beginning} at the base point. Readers only interested in this result are encouraged to skip to the bottom of Section \ref{subsub:boundary}.

\begin{rmk}
    For convenience, we make use of the language of $\infty$-categories when discussing homotopy coherent diagrams, i.e. a homotopy coherent diagram of spaces will equivalently be described as a ``commuting diagram in the $\infty$-category of spaces''.  We take Lurie \cite{Lurie:bh} as a primary reference.  We will use the language ``strictly commuting'' to indicate when we are working in a 1-category (e.g. of spaces), or by slight abuse of terminology, to denote a diagram of categories which commutes up to canonical equivalence.  We note that in general, a commuting diagram in an $\infty$-category encodes an infinite amount of data.  In practice, one constructs these by constructing a strictly commuting diagram in a 1-category (or a 2-commuting diagram in a 2-category) and then applying a functor to the desired $\infty$-category.  In this work, every commuting diagram in an $\infty$-category will arise in this fashion.
\end{rmk}

\subsubsection{The $K$-Theory Space of an Exact Category}\label{subsub:k}
\begin{definition}[Waldhausen]\label{defi:Waldhausen}
Let $\C$ be an exact category. Denote by $S_{\bullet}(\C)$ the simplicial object in exact categories defined as follows. The exact category of $n$-simplices $S_n(\C)$ is defined to be the exact category with objects given by strings of admissible monics in $\C$
$$(X_1 \hookrightarrow X_2 \hookrightarrow \cdots \hookrightarrow X_n)$$
along with choices of quotients $X_j/X_i$ for all $i<j$. The face maps are given by
$$d_i\colon (X_1 \hookrightarrow \cdots \hookrightarrow X_n) \mapsto (X_1 \hookrightarrow \cdots X_{i-1} \hookrightarrow X_{i+1} \hookrightarrow \cdots \hookrightarrow X_n),$$
for $i \geq 1$, and
$$d_0\colon (X_1 \hookrightarrow \cdots \hookrightarrow X_n) \mapsto (X_2/X_1 \hookrightarrow \cdots \hookrightarrow X_n/X_1).$$
The degeneracy maps $s_i\colon S_n(\C) \to S_{n+1}(\C)$ are defined by repeating the $i$-th entry. We refer to this simplicial object as \emph{Waldhausen's $S$-construction}.
\end{definition}

\begin{definition}\label{defi:modulispace}
    For a category $\C$ we denote by $\C^{\grp}$ the maximal sub-groupoid of $\C$, i.e. the groupoid obtained by discarding all non-invertible morphisms from the category $\C$.
\end{definition}

\begin{definition}\label{defi:Waldhausen_K}
    For an exact category $\C$ we define the $K$-theory space $K_{\C}$ as the loop space
    $$K_{\C} =Ê\Omega|S_{\bullet}(\C)^{\times}|$$
    where $|\cdot|$ denotes geometric realization.
\end{definition}

\begin{rmk}\label{rmk:natural}
    The simplicial object $S_\bullet(\C)^{\times}$ has the property that its space of $0$-simplices is a singleton. Hence, every $1$-simplex induces a loop in the geometric realization. Therefore, we have a map
    $$\C^{\grp} \cong S_{1}\C^{\grp} \to \Omega|S_\bullet(\C)^{\times}| = K_{\C},$$
    which is natural in $\C$.
\end{rmk}

\subsubsection{Additivity}
The fundamental property of algebraic $K$-theory is established in the following ``Additivity Theorem''. All of the results of this paper can be understood as consequences of the Additivity Theorem combined with Theorem \ref{thm:sato_filtered}.
\begin{theorem}[Waldhausen's Additivity Theorem]\label{thm:wald_add1}(\cite[Theorem 1.4.2, Proposition 1.3.2(4)]{MR0802796})
    Let $F_1\hookrightarrow F_2\twoheadrightarrow F_3$ be an exact sequence of functors $\C_1\to\C_2$. Then the map
    \begin{align*}
        |S_\bullet F_2|\colon |S_\bullet(\C_1)^\times|&\to|S_\bullet(\C_2)^\times|\intertext{is naturally homotopic to}
        |S_\bullet F_1\oplus S_\bullet F_3|\colon |S_\bullet(\C_1)^\times|&\to|S_\bullet(\C_2)^\times|.
    \end{align*}
\end{theorem}

Several equivalent reformulations exist. We will need the following one.
\begin{definition}[Waldhausen]
    Let $\D$ be an exact category, and let $\C_1$ and $\C_2$ be full sub-categories of $\D$ which are closed under extensions.
    Define $\E(\C_1,\D,\C_2)$ to be the full sub-category of $\E\D$ consisting of the exact sequences $X_1\hookrightarrow
    Y\twoheadrightarrow X_2$ with $X_i\in\C_i$.
\end{definition}
Note that, because $\C_1$ and $\C_2$ are closed under extensions in $\D$, $\E(\C_1,\D,\C_2)$ is closed under extensions in
$\E\D$; in particular, it is an exact category.

\begin{theorem}(\cite[Theorem 1.4.2, Proposition 1.3.2(1)]{MR0802796})\label{thm:wald_add2}
    The projection
    \begin{align*}
        |S_\bullet(\E(\C_1,\D,\C_2))^\times|&\to |S_\bullet(\C_1)^\times|\times |S_\bullet(\C_2)^\times|\\
                (X_1\hookrightarrow Y\twoheadrightarrow X_2)&\mapsto (X_1,X_2)
    \end{align*}
    is a homotopy equivalence.
\end{theorem}

\subsubsection{The $K$-Theory fiber Sequence}
A fundamental consequence of the Additivity Theorem is that an exact functor $\C\to^f\D$ determines a natural fiber sequence of $K$-theory spaces. We recall relevant details from \cite{MR0802796} here.

\begin{definition}
    Let $\C$ be an exact category. Define the \emph{right path-space} of $S_\bullet(\C)$ to be the simplicial diagram of exact
    categories $P^r S_\bullet(\C)$ with $n$-simplices
    \begin{equation*}
        P^r S_n(\C):=S_{n+1}(\C),
    \end{equation*}
    with the face map $d_i$ given by the face map $d_{i+1}$ of $S_\bullet(\C)$, and with the degeneracy map $s_i$ given by the
    degeneracy $s_{i+1}$ of $S_\bullet(\C)$.
\end{definition}

The face maps $d_0\colon S_{n+1}(\C)\to S_n(\C)$ determine a map of simplicial diagrams of exact categories
\begin{equation*}
    P^r S_\bullet(\C)\to S_\bullet(\C).
\end{equation*}

\begin{definition}(\cite[Definition 1.5.4]{MR0802796})\label{def:relwaldRIGHT}
    Let $\C\to^f\D$ be an exact functor. Define the simplicial diagram of exact categories $S^r_\bullet(f)$ to
    be the strict pullback
    \begin{equation*}
        \xymatrix{
          S^r_\bullet(f) \ar[d]_\delta \ar[r] & P^r S_\bullet(\D) \ar[d] \\
          S_\bullet(\C) \ar[r]^{f} & S_\bullet(\D)   }.
    \end{equation*}
    Explicitly, the $n$-simplices $S^r_n(f)$ consist of the the full sub-category of $S_n(\C)\times S_{n+1}(\D)$ on the
    objects
    \begin{equation*}
        \left(Y_1\hookrightarrow\cdots\hookrightarrow Y_n;X_1\hookrightarrow\cdots\hookrightarrow X_{n+1}\right)
    \end{equation*}
    such that
    \begin{equation*}
        \left(f(Y_1)\hookrightarrow \cdots\hookrightarrow f(Y_n)\right)=\left(X_2/X_1\hookrightarrow \cdots\hookrightarrow
        X_{n+1}/X_1\right)
    \end{equation*}
    for all $i\ge 1$. The face and degeneracy maps are the products of the face and degeneracy maps for $S_\bullet(\C)$ and $P^r S_\bullet(\Dc)$. The map $S^r_\bullet(f)\to^\delta S_\bullet(\C)$ is the projection onto the $S_\bullet(\C)$-factor.
\end{definition}

The Additivity Theorem implies the following.
\begin{proposition}(cf. the proof of \cite[Proposition 1.5.5]{MR0802796})\label{prop:rreladd}
    Let $\C\to^f\D$ be an exact map of exact categories. The map
    \begin{align*}
        S^r_n(f)&\to^{q^r} \D\times S_n(\C),\\
        \left(Y_1\hookrightarrow\cdots\hookrightarrow Y_n;X_1\hookrightarrow\cdots\hookrightarrow
        X_{n+1}\right)&\to/|-{>}/^{q^r}(X_1,Y_1\hookrightarrow\cdots\hookrightarrow Y_n)\intertext{induces an equivalence}
        |S_\bullet(S^r_n(f))^\times|&\to^\simeq |S_\bullet(\D)^\times|\times |S_\bullet (S_n(\C))^\times|.
    \end{align*}
\end{proposition}
\begin{proof}
    A right inverse to the above map $q^r$ is given by the map
    \begin{align*}
        S_n(\C)\times \D&\to^\sigma S^r_n(f)\intertext{sending}
        (X,Y_1\hookrightarrow\cdots\hookrightarrow Y_n)&\to/|-{>}/^\sigma \left(Y_1\hookrightarrow\cdots\hookrightarrow
        Y_n;X\hookrightarrow X\oplus f(Y_1)\hookrightarrow\cdots\hookrightarrow X\oplus f(Y_n)\right).
    \end{align*}
    It remains to exhibit a homotopy $|S_\bullet\sigma|\circ|S_\bullet q^r|\simeq 1_{|S_\bullet(S^r_n(f))|}$. For this, consider the functors
    \begin{align*}
        S^r_n(f)&\to^{\alpha} S^r_n(f)\\
        \left(Y_1\hookrightarrow\cdots\hookrightarrow Y_n;X_1\hookrightarrow\cdots\hookrightarrow
        X_{n+1}\right)&\to/|-{>}/^{\alpha} \left(0\hookrightarrow\cdots\hookrightarrow 0;X_1\hookrightarrow\cdots\hookrightarrow X_1\right)\intertext{and}
        S^r_n(f)&\to^{\beta} S^r_n(f)\\
        \left(Y_1\hookrightarrow\cdots\hookrightarrow Y_n;X_1\hookrightarrow\cdots\hookrightarrow
        X_{n+1}\right)&\to/|-{>}/^{\beta} \left(Y_1\hookrightarrow\cdots\hookrightarrow Y_n;0\hookrightarrow f(Y_1)\hookrightarrow\cdots\hookrightarrow f(Y_n)\right).
    \end{align*}
    Then we have a natural isomorphism $\sigma\circ q^r\cong \alpha\oplus\beta$ as well as a short exact sequence of functors
    \begin{equation*}
        \alpha\hookrightarrow 1_{S^r_n(f)}\twoheadrightarrow\beta.
    \end{equation*}
    By the Additivity Theorem (Theorem \ref{thm:wald_add1}), there exists a homotopy
    \begin{equation*}
        |S_\bullet\sigma|\circ |S_\bullet q^r|\simeq |S_\bullet 1_{S^r_n(f)}|.
    \end{equation*}
    We see that $|S_\bullet q^r|$ is a homotopy equivalence as claimed.
\end{proof}

Let $\D_\bullet^{\triv}$ denote the constant simplicial diagram on $\D$. The identity map $\D\to^1\D$ extends to an exact map of
simplicial diagrams of exact categories
\begin{equation*}
    \D^{\triv}_\bullet\to S^r_\bullet(f).
\end{equation*}
Applying the $S$-construction to this map, we obtain an exact map of bisimplicial diagrams of exact categories
\begin{equation*}
    S_\bullet(\D)^{v-\triv}_\bullet\to^q S_\bullet S^r_\bullet(f),
\end{equation*}
where the superscript ``$v-\triv$'' indicates that the bisimplicial object is constant in the vertical direction.\footnote{We adopt the convention that in a bisimplicial set $X_{\bullet,\bullet}$, viewed as a first quadrant diagram, the first bullet denotes the horizontal coordinate, while the second denotes the vertical one.} Applying the
$S$-construction to the map
\begin{equation*}
    S^r_\bullet(f)\to^\delta S_\bullet(\C),
\end{equation*}
we obtain an exact map of bisimplicial diagrams of exact categories
\begin{equation*}
    S_\bullet S^r_\bullet(f)\to^{S_\bullet\delta} S_\bullet S_\bullet(\C).
\end{equation*}
The above maps determine a strictly commuting square
\begin{equation}\label{wald155square}
    \xymatrix{
        (S_\bullet(\D)^\times)^{v-\triv}_\bullet \ar[r] \ar[d] & S_\bullet(S^r_\bullet(f))^\times \ar[d]^{S_\bullet\delta}\\
        \ast \ar[r] & S_\bullet (S_\bullet(\C))^\times.
    }
\end{equation}

Proposition \ref{prop:rreladd} provides the core of the proof of the following.
\begin{proposition}(\cite[Proposition 1.5.5]{MR0802796})\label{prop:wald155}
    The geometric realization of the square \eqref{wald155square} is a homotopy pullback square, equivalently a cartesian square in the $\infty$-category of spaces.
\end{proposition}

\begin{corollary}(\cite[Corollary 1.5.6]{MR0802796})
    Let $\C_1\to^f\C_2\to^g\C_3$ be a sequence of exact functors. Then the strictly commuting square
    \begin{equation*}
        \xymatrix{
            |S_\bullet(\C_2)^\times| \ar[r] \ar[d] & |S_\bullet(\C_3)^\times| \ar[d] \\
            |S_\bullet(S^r_\bullet(f))^\times| \ar[r] & |S_\bullet(S^r_\bullet(gf))^\times|
        }
    \end{equation*}
    is a homotopy pullback, i.e. a cartesian square in the $\infty$-category of spaces.
\end{corollary}

If we consider the sequence $\C\to^1\C\to^f\D$, we obtain the following.
\begin{corollary}(\cite[Corollary 1.5.7]{MR0802796})\label{cor:wald157}
    Let $\C\to^f\D$ be an exact functor. Then the strictly commuting square
    \begin{equation}\label{r157square}
        \xymatrix{
            |S_\bullet(\C)^\times| \ar[d] \ar[r]^{f} & |S_\bullet(\D)^\times| \ar[d] \\
            |S_\bullet S^r_\bullet(1_{\C})^\times| \ar[r] & |S_\bullet S^r_\bullet(f)^\times|   }.
    \end{equation}
    is a homotopy pullback, i.e. a cartesian square in the $\infty$-category of spaces.
\end{corollary}

\begin{lemma}\label{lemma:pathcontract}
    Let $\C$ be an exact category. Then $|S^r_\bullet(1_\C)^\times|=|P^r S_\bullet(\C)^\times|\simeq\ast$.
\end{lemma}
\begin{proof}
    Denote by $S_\bullet(\C)^\times_\bullet$ the bisimplicial set obtained by taking the nerves of the groupoids $S_n(\C)^\times$ in the vertical direction, and denote by $P^r S_{\bullet}(\C)^\times_\bullet$ the analogue for $P^r S_\bullet(\C)^\times$. For each $m$, the horizontal simplicial set $P^r S_\bullet(\C)^\times_m$ is obtained from $S_\bullet(\C)^\times_m$ by forgetting the zeroth face and degeneracy maps and shifting all simplicial indices down by one (i.e. $P^r S_n(\C)^\times_m=S_{n+1}(\C)^\times_m$ and the $i^{th}$ face and degeneracy maps are given by the maps $d_{i+1}$ and $s_{i+1}$ on $S_{n+1}(\C)^\times$). Recall (e.g. from \cite[p. 219]{MR1897816}) that the maps
    \begin{equation*}
        d_0^{n+1}\colon P^r S_n(\C)^\times_m\leftrightarrows S_0(\C)^\times_m=\ast\colon s_0^{n+1}
    \end{equation*}
    are the value on $n$-simplices of a homotopy equivalence
    \begin{equation*}
        P^r S_\bullet(\C)^\times_m\simeq \ast.
    \end{equation*}
    Because geometric realization preserves level-wise weak equivalences, we conclude that $|P^r S_\bullet(\C)^\times|\simeq\ast$ is a weak equivalence.
\end{proof}

\begin{proposition}[Waldhausen]\label{prop:wald157}
    Let $\C\to^f\D$ be an exact functor. The homotopy equivalence of Lemma \ref{lemma:pathcontract} determines a homotopy which makes the square
    \begin{equation}\label{rWaldsquare}
        \xymatrix{
            |S_\bullet(\C)^\times| \ar[d] \ar[r]^{f} & |S_\bullet(\D)^\times| \ar[d] \\
            \ast \ar[r] & |S_\bullet S^r_\bullet(f)^\times|
        }
    \end{equation}
    homotopy commute and a homotopy pullback.
\end{proposition}

If we take $\D$ to be the zero category, we see that the square \eqref{rWaldsquare} induces the inclusion of 1-simplices map
\begin{equation*}
    |S_\bullet(\C)^\times|\to\Omega|S_\bullet S_\bullet(\C)^\times|
\end{equation*}
of Remark \ref{rmk:natural}. The proposition shows that this is in fact an equivalence. By iterating the $S$-construction, one sees that the $K$-theory space $K_{\C}$ of an exact category $\C$ is canonically an \emph{infinite loop space}. We denote the corresponding connective spectrum by $\Kk_{\C}$.

\begin{rmk}
    For an exact category $\C$, the connective spectrum $\Kk_{\C}$ admits a natural \emph{non-connective} variant $\Kb_{\C}$, capturing negative $K$-theory groups. The constructions of the present paper extend to non-connective $K$-theory and similar invariants of exact categories. We leave this to future work.
\end{rmk}

We conclude this paragraph by introducing a dual version of the relative $S$-construction, which plays a role in the applications below.
\begin{definition}\label{def:Spath}
    Let $\C$ be an exact category. Define the \emph{left path-space} of $S_\bullet(\C)$ to be the simplicial diagram of exact
    categories $P^\ell S_\bullet(\C)$ with $n$-simplices
    \begin{equation*}
        P^\ell S_n(\C):=S_{n+1}(\C),
    \end{equation*}
    with the face map $d_i$ given by the face map $d_i$ of $S_\bullet(\C)$, and with the degeneracy map $s_i$ given by the
    degeneracy $s_i$ of $S_\bullet(\C)$.
\end{definition}

The face maps $d_{n+1}\colon S_{n+1}(\C)\to S_n(\C)$ determine a map of simplicial exact categories
\begin{equation*}
    P^\ell S_\bullet(\C)\to S_\bullet(\C).
\end{equation*}

\begin{definition}\label{def:relwald}
    Let $\C\to^f\D$ be an exact map of exact categories. Define the simplicial diagram of exact categories $S^\ell_\bullet(f)$
    to be the strict pullback
    \begin{equation*}
        \xymatrix{
          S^\ell_\bullet(f) \ar[d]_\delta \ar[r] & P^\ell S_\bullet(\D) \ar[d] \\
          S_\bullet(\C) \ar[r]^{f} & S_\bullet(\D)   }.
    \end{equation*}
    Explicitly, the $n$-simplices $S^\ell_n(f)$ consist of the the full sub-category of $S_n(\C)\times S_{n+1}(\D)$ on the
    objects
    \begin{equation*}
        \left(Y_1\hookrightarrow\cdots\hookrightarrow Y_n;X_1\hookrightarrow\cdots\hookrightarrow X_{n+1}\right)
    \end{equation*}
    such that
    \begin{equation*}
        \left(f(Y_1)\hookrightarrow \cdots\hookrightarrow f(Y_n)\right)=\left(X_1\hookrightarrow \cdots\hookrightarrow
        X_n\right)
    \end{equation*}
    for all $i\ge 1$. The face and degeneracy maps are the products of the face and degeneracy maps for $S_\bullet(\C)$ and
    $P^\ell S_\bullet(\Dc)$.  The map $S^\ell_\bullet(f)\to^\delta S_\bullet(\C)$ is the projection onto the $S_\bullet(\C)$
    factor.
\end{definition}

\begin{lemma}\label{lemma:righttoleft}
    Let $\C\to^f\D$ be an exact functor. Denote by $f^{\op}$ the induced functor on opposite categories, and denote by $t\colon\Ddelta\to\Ddelta$ the functor which sends a finite ordinal to its opposite (equivalently, $\mathit{t}^\ast$ reverses the order of simplices in a simplicial diagram). Then there exists a natural equivalence of simplicial diagrams of exact categories
    \begin{equation*}
        \mathit{t}^\ast S^r_\bullet(f^{\op})^{\op}\to^\simeq S^\ell_\bullet(f).
    \end{equation*}
    We emphasize that on the left hand side, we have first replaced $f$ by $f^{\op}$, then applied $S^r_\bullet(-)$, then taken the opposite categories in the simplicial diagram $S^r_\bullet(f^{\op})$ and then reversed the orientation of the simplices in this diagram.
\end{lemma}
\begin{proof}
    We begin by observing that $S^r_n(f^{\op})^{\op}$ is equivalent to the category consisting of objects
    \begin{equation*}
        (\bar{Y};\bar{X};\varphi):=\left(Y_n\onto\cdots\onto Y_1;X_{n+1}\onto\cdots\onto X_1;\varphi\right)
    \end{equation*}
    where $Y_n\onto\cdots\onto Y_1$ is a string of admissible epics in $\C$, $X_{n+1}\onto\cdots\onto X_1$ is a string of admissible epics in $\D$, and $\varphi$ is an isomorphism
    \begin{equation*}
        (f(Y_n)\onto\cdots\onto f(Y_1))\to^\cong(\ker(X_{n+1}\onto X_1)\onto\cdots\onto\ker(X_2\onto X_1)).
    \end{equation*}
    A morphism $(\bar{Y}^0;\bar{X}^0;\varphi^0)\to(\bar{Y}^1;\bar{X}^1;\varphi^1)$ consists of a collection of morphisms
    \begin{align*}
        Y^0_i&\to Y^1_i\\
        X^0_i&\to X^1_i
    \end{align*}
    making all of the appropriate diagrams commute.

    Consider the assignment which sends $\left(Y_n\onto\cdots\onto Y_1;X_{n+1}\onto\cdots\onto X_1;\varphi\right)$ to
    \begin{align*}
        (&\ker(Y_n\onto Y_{n-1})\into\cdots\into\ker(Y_n\onto Y_1)\into Y_n;\\
        &\ker(X_{n+1}\onto X_n)\into\cdots\into\ker(X_{n+1}\onto X_1)\into X_{n+1};\tilde{\varphi}),
    \end{align*}
    where $\tilde{\varphi}$ denotes the isomorphism
    \begin{align*}
        &\left(\ker(Y_n\onto Y_{n-1})\into\cdots\into\ker(Y_n\onto Y_1)\right)\\
        &\cong\left(\ker(X_{n+1}\onto X_n)\into\cdots\into\ker(X_{n+1}\onto X_1)\right)
    \end{align*}
    induced, by Noether's lemma, from $\varphi$. This extends to an equivalence of categories
    \begin{equation*}
        S^r_n(f^{\op})^{\op}\to^\simeq S^\ell_n(f),
    \end{equation*}
    where the inverse is defined in the analogous manner.

    Under this equivalence, the face map $d_i$ on $S^r_n(\C^{\op}\subset\D^{\op})^{\op}$ corresponds to the face map $d_{n-i}$ on     $S^\ell_n(\C\subset\D)$, while the degeneracy $s_i$ corresponds to the degeneracy $s_{n-i}$. Letting $n$ vary, these
    equivalences determine an equivalence of simplicial diagrams of exact categories.
\end{proof}

\begin{rmk}\label{rmk:rtol}
    Note that the equivalence of Lemma \ref{lemma:righttoleft} fits into a natural commuting square
    \begin{equation*}
        \xymatrix{
            (\D^{\op})^{\op} \ar[r] \ar[d]_1 & \mathit{t}^\ast S^r_\bullet(f^{\op})^{\op} \ar[d]^\simeq\\
            \D \ar[r] & S^\ell_\bullet(f).
        }
    \end{equation*}
\end{rmk}

Taking $\D=0$ in the lemma above, we obtain the following.
\begin{corollary}\label{cor:babyrtol}
    There is a natural equivalence $\mathit{t}^\ast S_\bullet(\C^{\op})^{\op}\simeq S_\bullet(\C)$.
\end{corollary}

Combining this with the results above, we obtain the following.
\begin{proposition}\label{prop:Slprops}
    Let $\C\to^f\D$ be an exact functor.
    \begin{enumerate}
        \item[(a)] The map
            \begin{align*}
                S^\ell_n(f)&\to^{q^\ell} S_n(\C)\times\D,\intertext{given on objects by the assignment}
                \left(Y_1\hookrightarrow\cdots\hookrightarrow Y_n;X_1\hookrightarrow\cdots\hookrightarrow
                X_{n+1}\right)&\to/|-{>}/^{q^\ell}(Y_1\hookrightarrow\cdots\hookrightarrow Y_n,X_{n+1}/X_n)\intertext{induces a homotopy
                equivalence}
                |S_\bullet(S^\ell_n(f))^\times|&\to^\simeq |S_\bullet S_n(\C)^\times|\times |S_\bullet(\D)^\times|.
            \end{align*}
        \item[(b)] There exists a natural strictly commuting cube
            \begin{equation}\label{rtolcube}
                \xymatrix@=9pt{
                    |S_\bullet(\C^{\op})^\times| \ar[rr]^{f^{\op}} \ar[dr]^\simeq \ar[dd] && |S_\bullet(\D^{\op})^\times| \ar[dr]^\simeq \ar'[d][dd]\\
                    & |S_\bullet(\C)^\times| \ar[dd] \ar[rr]^(.3){f} && |S_\bullet(\D)^\times| \ar[dd] \\
                    |S_\bullet S^r_\bullet(1_{\C^{\op}})^\times| \ar'[r][rr] \ar[dr]^\simeq && |S_\bullet S^r_\bullet(f^{\op})^\times| \ar[dr]^\simeq \\
                    & |S_\bullet S^\ell_\bullet(1_{\C})^\times| \ar[rr] && |S_\bullet S^\ell_\bullet(f)^\times|.
                }
            \end{equation}
            in which all the diagonal arrows are equivalences.  In particular, the front face is a homotopy pullback, and its lower left corner is contractible.
    \end{enumerate}
\end{proposition}
\begin{proof}
    For the first statement, the definition of the map $q^\ell$ ensures that, under the equivalences of Lemma \ref{lemma:righttoleft}, it is naturally isomorphic to the map
    \begin{equation*}
        S^r_n(f^{\op})^{\op}\to^{q^{r,\op}} S_n(\C^{\op})^{\op}\times(\D^{\op})^{\op},
    \end{equation*}
    where $q^r$ is the map associated to $\C^{\op}\to^{f^{\op}}\D^{\op}$ in Proposition \ref{prop:rreladd}. Applying the $S$-construction, we obtain a map
    \begin{equation*}
        S_\bullet(S^r_n(f^{\op})^{\op})\to^{S_\bullet q^{r,\op}} S_\bullet(S_n(\C^{\op})^{\op})\times S_\bullet(\D).
    \end{equation*}
    By Corollary \ref{cor:babyrtol}, this map is naturally equivalent to
    \begin{equation*}
        \mathit{t}^\ast S_\bullet(S^r_n(f^{\op}))^{\op}\to^{\mathit{t}^\ast S_\bullet q^r} \mathit{t}^\ast S_\bullet(S_n(\C^{\op})\times \D^{\op})^{\op}
    \end{equation*}
    For any groupoid $\mathcal{G}$, the assignment (on morphisms) $g\mapsto g^{-1}$ determines a natural equivalence of groupoids    $\mathcal{G}\to^\simeq\mathcal{G}^{\op}$. There is also a canonical natural equivalence $|\mathit{t}^\ast X|\to^\simeq|X|$ for any simplicial set $X_\bullet$. Applying both of these equivalences, we see that the previous map fits into a commuting square of spaces
    \begin{equation*}
        \xymatrix{
            |\mathit{t}^\ast (S_\bullet(S^r_n(f^{\op}))^{\op})^\times| \ar[d]_\simeq \ar[rr]^{|\mathit{t}^\ast S_\bullet q^r|} && |\mathit{t}^\ast (S_\bullet(S_n(\C^{\op})\times \D^{\op})^{\op})^\times| \ar[d]^\simeq\\
            |S_\bullet(S^r_n(f^{\op}))^\times| \ar[rr] && |S_\bullet(S_n(\C^{\op})\times\D^{\op})^\times|.
        }
    \end{equation*}
    Applying the canonical equivalence $|S_\bullet(S_n(\C^{\op})\times\D^{\op})^\times|\to^\simeq |S_\bullet(S_n(\C^{\op}))^\times|\times |S_\bullet(\D^{\op})^\times|$, we see that the bottom map is the equivalence of Proposition \ref{prop:rreladd}.

    The second statement follows by a similar argument. The equivalences $\mathcal{G}\simeq \mathcal{G}^{\op}$ and $|X|\simeq|\mathit{t}^\ast X|$ determine a commuting cube
    \begin{equation*}
        \xymatrix@=9pt{
            |S_\bullet(\C^{\op})^\times| \ar[rr]^{f^{\op}} \ar[dr]^\simeq \ar[dd] && |S_\bullet(\D^{\op})^\times| \ar[dr]^\simeq \ar'[d][dd]\\
            & |\mathit{t}^\ast (S_\bullet(\C^{\op})^{\op})^\times| \ar[dd] \ar[rr]^(.3){f^{\op}} && |\mathit{t}^\ast (S_\bullet(\D^{\op})^{\op})^\times| \ar[dd] \\
            |S_\bullet S^r_\bullet(1_{\C^{\op}})^\times| \ar'[r][rr] \ar[dr]^\simeq && |S_\bullet S^r_\bullet(f^{\op})^\times| \ar[dr]^\simeq \\
            & |\mathit{t}^\ast (S_\bullet(S^r_\bullet(1_{\C^{\op}}))^{\op})^\times| \ar[rr] && |\mathit{t}^\ast (S_\bullet (S^r_\bullet(f^{\op}))^{\op})^\times|.
        }
    \end{equation*}
    Applying the equivalence of Corollary \ref{cor:babyrtol} to the outer $S_\bullet$, we see that the front face of this square fits into a commuting cube
    \begin{equation*}
        \xymatrix@=9pt{
            |\mathit{t}^\ast (S_\bullet(\C^{\op})^{\op})^\times| \ar[dd]\ar[dr]^\simeq \ar[rr]^{f^{\op}} && |\mathit{t}^\ast (S_\bullet(\D^{\op})^{\op})^\times| \ar'[d][dd] \ar[dr]^\simeq \\
            & |S_\bullet(\C)^\times| \ar[rr]^(.3){f} \ar[dd] && |S_\bullet(\D)^\times| \ar[dd]\\
            |\mathit{t}^\ast (S_\bullet(S^r_\bullet(1_{\C^{\op}}))^{\op})^\times| \ar[dr]^\simeq \ar'[r][rr] && |\mathit{t}^\ast (S_\bullet(S^r_\bullet(f^{\op}))^{\op})^\times| \ar[dr]^\simeq \\
            & |S_\bullet(S^r_\bullet(1_{\C^{\op}})^{\op})^\times| \ar[rr] && |S_\bullet(S^r_\bullet(f^{\op})^{\op})^\times|.
        }
    \end{equation*}
    Applying the equivalence of Lemma \ref{lemma:righttoleft} to the inner $S$-construction, we obtain a commuting cube
    \begin{equation*}
        \xymatrix@=9pt{
            |S_\bullet(\C)^\times| \ar[rr]^f \ar[dd] \ar[dr]^1 && |S_\bullet(\D)^\times| \ar'[d][dd] \ar[dr]^1 \\
            & |S_\bullet(\C)^\times| \ar[rr]^(.3){f} \ar[dd] && |S_\bullet(\D)^\times| \ar[dd] \\
            |S_\bullet(S^r_\bullet(1_{\C^{\op}})^{\op})^\times| \ar[rr] \ar[dr]^\simeq && |S_\bullet(S^r_\bullet(f^{\op})^{\op})^\times| \ar[dr]^\simeq\\
            & |S_\bullet(\mathit{t}^\ast S^\ell_\bullet(1_{\C}))^\times| \ar[rr] && |S_\bullet(\mathit{t}^\ast S^\ell_\bullet(f))^\times|.
        }
    \end{equation*}
    By a final application of the equivalence $|X|\simeq |\mathit{t}^\ast X|$, we see that the front face of this cube is equivalent to the front face of \eqref{rtolcube}. Composing this equivalence with the cubes above, we obtain the cube \eqref{rtolcube} as claimed.
\end{proof}

\subsubsection{The Localization Sequence for Exact Categories}
In \cite{MR2079996}, Schlichting established a fundamental ``Localization Theorem'' for the $K$-theory of exact categories.
\begin{proposition}(Schlichting \cite[Lemma 2.3]{MR2079996})\label{prop:Schlichting1}
    Let $\C\subset\D$ be the inclusion of an idempotent complete, right s-filtering sub-category. Consider the map of simplicial diagrams of categories
    \begin{equation*}
        S^r_\bullet(\C\subset\D)\to^q \D/\C
    \end{equation*}
    given by the assignment
    \begin{equation*}
        \left(Y_1\hookrightarrow\cdots\hookrightarrow Y_n;X_1\hookrightarrow \cdots\hookrightarrow
        X_{n+1}\right)\mapsto X_{n+1}.
    \end{equation*}
    Then all of the diagonal arrows are equivalences in the commuting cube
    \begin{equation*}
        \xymatrix@=9pt{
            |S_\bullet(\C)^\times| \ar[dr]^1 \ar[rr]^f \ar[dd] && |S_\bullet(\D)^\times| \ar'[d][dd] \ar[dr]^1 \\
            & |S_\bullet(\C)^\times| \ar[rr]^(.3){f} \ar[dd] && |S_\bullet(\D)^\times| \ar[dd]\\
            |S_\bullet S^r_\bullet(1_{\C})^\times| \ar'[r][rr] \ar[dr] && |S_\bullet S^r_\bullet(f)^\times| \ar[dr]^{|S_\bullet q|} \\
            & \ast \ar[rr] && |S_\bullet(\D/\C)^\times|. }
    \end{equation*}
\end{proposition}

Combined with Proposition \ref{prop:Slprops}, this implies the following.
\begin{proposition}\label{prop:Schlichtcube}
    Let $\C\subset\D$ be the inclusion of an idempotent complete, left s-filtering sub-category. Consider the map of simplicial diagrams of categories
    \begin{equation*}
        S^\ell_\bullet(\C\subset\D)\to^q \D/\C
    \end{equation*}
    given by the assignment
    \begin{equation*}
        \left(Y_1\hookrightarrow\cdots\hookrightarrow Y_n;X_1\hookrightarrow\cdots\hookrightarrow
        X_{n+1}\right)\mapsto X_{n+1}.
    \end{equation*}
    Then all of the diagonal arrows are equivalences in the homotopy coherent cube (i.e. commuting cube in the $\infty$-category of spaces)
    \begin{equation*}
        \xymatrix@=9pt{
            |S_\bullet(\C)^\times| \ar[dr]^1 \ar[rr]^f \ar[dd] && |S_\bullet(\D)^\times| \ar'[d][dd] \ar[dr]^1 \\
            & |S_\bullet(\C)^\times| \ar[rr]^(.3){f} \ar[dd] && |S_\bullet(\D)^\times| \ar[dd]\\
            |S_\bullet S^\ell_\bullet(1_{\C})^\times| \ar'[r][rr] \ar[dr] && |S_\bullet S^\ell_\bullet(f)^\times| \ar[dr]^{|S_\bullet q|} \\
            & \ast \ar[rr] && |S_\bullet(\D/\C)^\times|. }
    \end{equation*}
\end{proposition}
\begin{proof}
    Because $\C\subset\D$ is left s-filtering, $\C^{\op}\subset\D^{\op}$ is right s-filtering. We can now compose the cube of Proposition \ref{prop:Schlichting1} with the cube \eqref{rtolcube} of Proposition \ref{prop:Slprops} to obtain a cube of the form above. By tracing through the construction, we see that the equivalence in this cube
    \begin{equation*}
        |S_\bullet S^\ell_\bullet(\C\subset\D)^\times|\to^\simeq |S_\bullet(\D/\C)^\times|
    \end{equation*}
    is given by the map $|S_\bullet q|$ above.
\end{proof}

Combined with Propositions \ref{prop:wald157} and \ref{prop:Slprops}, these propositions give the following.
\begin{theorem}[Schlichting's Localization Theorem]\label{thm:schlichting}
    Let $\C\subset\D$ be the inclusion of an idempotent complete, left or right s-filtering sub-category. Then the square
    \begin{equation}\label{schlict}
        \xymatrix{
            K_{\C} \ar[r] \ar[d] & K_{\D} \ar[d] \\
            \ast \ar[r] & K_{\D/\C}
        }
    \end{equation}
    is a homotopy pullback.
\end{theorem}

\subsubsection{Boundary Maps in Algebraic $K$-Theory}\label{subsub:boundary}
By the universal property of homotopy pullbacks, the Localization Theorem associates, to a left or right s-filtering sub-category $\C\subset\D$, a boundary map
\begin{equation*}
    \Omega K_{\D/\C}\to^\partial K_{\C}.
\end{equation*}
Explicitly, let $q\colon \D\to\D/\C$ denote the canonical exact functor, and also the map on $K$-theory. For a based space $X$, let $PX$ denote the space of paths beginning at the base point, and let $p\colon PX\to X$ send a path to its endpoint. The inclusion of the constant path $\ast\to PK_{\D/\C}$ induces a canonical homotopy coherent cube (i.e. commuting cube in the $\infty$-category of spaces)
\begin{equation*}
    \xymatrix@=9pt{
        K_{\C} \ar[dr]^\simeq \ar[rr]^f \ar[dd] && K_{\D} \ar'[d][dd] \ar[dr]^1 \\
        & PK_{\D/\C}\times_{K_{\D/\C}} K_{\D} \ar[rr] \ar[dd] && K_{\D} \ar[dd]\\
        \ast \ar'[r][rr] \ar[dr] && K_{\D/\C} \ar[dr]^{1} \\
        & PK_{\D/\C}  \ar[rr] && K_{\D/\C}. }
\end{equation*}
in which the diagonal arrows are all equivalences. By the universal property of homotopy pullbacks, the front face determines a contractible space of maps
\begin{equation}\label{pindown}
    \Omega K_{\D/\C}=PK_{\D/\C}\times_{K_{\D/\C}}\ast\to  PK_{\D/\C}\times_{K_{\D/\C}} K_{\D}.
\end{equation}
The space of homotopy inverses of the homotopy equivalence
\begin{equation*}
    K_{\C}\to^\simeq PK_{\D/\C}\times_{K_{\D/\C}} K_{\D}
\end{equation*}
is contractible. Therefore, up to a contractible space of choices, the map \eqref{pindown} determines a map
\begin{equation*}
    \Omega K_{\D/\C}\to^\partial K_{\C}.
\end{equation*}
We refer to this as \emph{the} boundary map in the localization sequence.

For our applications, we will need to be able to describe this map in some detail. We therefore use this section to explain how the results of Waldhausen which we recalled above lead to an explicit description of this boundary map.

The proof of the Localization Theorem gives a canonical equivalence from the $K$-theory localization sequence to the looping of the Waldhausen fibration sequence \eqref{rWaldsquare}. To avoid a proliferation of $\Omega$s on the page, we will describe the boundary map
\begin{equation}\label{rWaldboundmap}
    \Omega|S_\bullet S^r_\bullet(f)^\times|\to^\partial|S_\bullet(\C)^\times|.
\end{equation}
Note that the boundary map in $K$-theory is obtained by applying $\Omega$ to this map.  Note also that, by Proposition \ref{prop:Slprops}, our description will immediately imply an analogous description of the map
\begin{equation*}\label{lWaldboundmap}
    \Omega|S_\bullet S^\ell_\bullet(f)^\times|\to^\partial|S_\bullet(\C)^\times|.
\end{equation*}

\begin{proposition}\label{prop:waldbound}
    Let $\C\to^f\D$ be an exact functor. The boundary map \eqref{rWaldboundmap} fits into a canonical homotopy commuting triangle (i.e. commuting triangle in the $\infty$-category of spaces)
    \begin{equation}\label{Waldbounddiagram}
        \xymatrix{
            \Omega|S_\bullet S^r_\bullet(f)^\times| \ar[drr]_\partial \ar[rr]^{\Omega|S_\bullet\delta|} && \Omega|S_\bullet
            S_\bullet(\C)^\times| \ar[d]^\simeq \\
            && |S_\bullet(\C)^\times|.
        }
    \end{equation}
    The equivalence in the triangle is inverse to the equivalence induced by the inclusion of vertical 1-simplices $S_\bullet(\C)\to^1 S_\bullet S_1(\C)$.
\end{proposition}
\begin{proof}
    By Corollary \ref{cor:wald157}, the commuting square of exact functors
    \begin{equation*}
        \xymatrix{
            \C \ar[r]^1 \ar[d]_f & \C \ar[d] \\
            \D \ar[r] & 0
        }
    \end{equation*}
    determines a commuting cube of spaces
    \begin{equation*}
        \xymatrix@=9pt{
            |S_\bullet(\C)^\times| \ar[rr]^1 \ar[dr]^f \ar[dd] && |S_\bullet(\C)^\times| \ar[dr] \ar'[d][dd] \\
            & |S_\bullet(\D)^\times| \ar[rr] \ar[dd] && \ast \ar[dd] \\
            |S_\bullet S^r_\bullet(1_{\C})^\times| \ar'[r][rr]^(.3){1} \ar[dr] && |S_\bullet S^r_\bullet(1_{\C})^\times|
            \ar[dr] \\
            & |S_\bullet S^r_\bullet(f)^\times| \ar[rr]^{|S_\bullet\delta|} && |S_\bullet S_\bullet(\C)^\times|
        }
    \end{equation*}
    in which the left and right faces are homotopy pullbacks. By Lemma \ref{lemma:pathcontract}, the lower rear corners of this diagram are contractible.  Therefore, by the universal property of homotopy pullbacks, this diagram determines a canonical homotopy commuting square
    \begin{equation*}
        \xymatrix{
            \Omega|S_\bullet S^r_\bullet(f)^\times| \ar[d]_\partial \ar[rr]^{\Omega|S_\bullet\delta|} && \Omega|S_\bullet
            S_\bullet(\C)^\times|  \ar[d]^\simeq \\
            |S_\bullet(\C)^\times| \ar[rr]^1 && |S_\bullet(\C)^\times|
        }
    \end{equation*}
    or equivalently, a homotopy commuting triangle as in \eqref{Waldbounddiagram}. As Waldhausen observed in \cite[Lemma
    1.5.2]{MR0802796}, the equivalence induced by the right face of the above cube is inverse to the equivalence
    $|S_\bullet(\C)^\times|\to^\simeq\Omega|S_\bullet S_\bullet(\C)^\times|$ induced by the inclusion of vertical 1-simplices
    $S_\bullet(\C)\to^1 S_\bullet S_1(\C)$.
\end{proof}

Combining the above, we obtain the following.
\begin{proposition}\label{prop:Schlichting}
    If $\C\subset\D$ is the inclusion of an idempotent complete, right s-filtering sub-category, then there exists a canonical homotopy commuting square (i.e. commuting square in the $\infty$-category of spaces)
    \begin{equation*}
        \xymatrix{
            \Omega|S_\bullet S^r_\bullet(\C\subset\D)^\times| \ar[d]_\simeq \ar[rr] && |S_\bullet(\C)^\times| \ar[d]^1\\
            \Omega|S_\bullet(\D/\C)^\times| \ar[rr]^\partial && |S_\bullet(\C)^\times|
        }
    \end{equation*}
    where the bottom map is the boundary map associated of the localization sequence, and the top map is the composition
    \begin{equation*}
        \Omega|S_\bullet S^r_\bullet(\C\subset\D)^\times|\to^{\Omega|S_\bullet \delta|} \Omega|S_\bullet S_\bullet(\C))^\times|\to^\simeq |S_\bullet(\C)^\times|.
    \end{equation*}
    Similarly, by tracing through the proof of the above, we see that if $\C\subset\D$ is the inclusion of an idempotent complete, left s-filtering sub-category, then there exists a canonical homotopy commuting square
    \begin{equation*}
        \xymatrix{
            \Omega|S_\bullet S^\ell_\bullet(\C\subset\D)^\times| \ar[d]_\simeq \ar[rr] && |S_\bullet(\C)^\times| \ar[d]^{1}\\
            \Omega|S_\bullet(\D/\C)^\times| \ar[rr]^{-\partial} && |S_\bullet(\C)^\times|
        }
    \end{equation*}
    where the top map is the composition
    \begin{equation*}
        \Omega|S_\bullet S^\ell_\bullet(\C\subset\D)^\times|\to^{\Omega|S_\bullet \delta|} \Omega|S_\bullet S_\bullet(\C)^\times|\to^\simeq |S_\bullet(\C)^\times|.
    \end{equation*}
\end{proposition}

\section{The Index Map}\label{index}
In this section, we use Theorem \ref{thm:sato_filtered} to define the index map $\Omega K_{\elTate(\C)}\to^{\Index} K_{\C}$. In Theorem \ref{thm:grsec} and Corollary \ref{cor:explicitindex}, we produce an explicit combinatorial model for this map. Using the Additivity Theorem, we show in Theorem \ref{thm:deloop} that the index map is an equivalence. We conclude this section with Theorem \ref{thm:loc}, where we relate the index map to the boundary map associated by the Localization Theorem to the left s-filtering inclusion $\C\subset\Ind(\C)$.

\subsection{The Categorical Index Map}
For $n\ge 0$, denote by $[n]$ the partially ordered set $\{0<\ldots<n\}$, and denote by $\Fun([n],\C)$ the category of functors
from the partially ordered set $[n]$, viewed as a category, to a category $\C$.

\begin{definition}
    Let $\C$ be an exact category. Define the \emph{Sato complex} $\Gr_\bullet^\le(\C)$ to be the simplicial diagram of exact
    categories with
    \begin{enumerate}
        \item $n$-simplices $\Gr_n^\le(\C)$ given by the full sub-category of $\Fun([n+1],\elTate(\C))$ consisting of
            sequences
            of admissible monics
            \begin{equation*}
                L_0\hookrightarrow\cdots\hookrightarrow L_n\hookrightarrow V
            \end{equation*}
            where, for all $i$, $L_i\hookrightarrow V$ is the inclusion of a lattice,\footnote{To see that this is an exact
            category, observe that because $\Pro(\C)$ and $\Ind(\C)$ are closed under extensions in $\elTate(\C)$,
            $\Gr_n^\le(\C)$ is closed under extensions in $\Fun([n+1],\elTate(\C))$.}
        \item face maps are given by the functors
            \begin{equation*}
                d_i(L_0\hookrightarrow\cdots\hookrightarrow L_n\hookrightarrow V):=(L_0\hookrightarrow\cdots\hookrightarrow
                L_{i-1}\hookrightarrow L_{i+1}\hookrightarrow\cdots\hookrightarrow L_n\hookrightarrow V),
            \end{equation*}
        \item and degeneracy maps are given by the functors
            \begin{equation*}
                s_i(L_0\hookrightarrow\cdots\hookrightarrow L_n\hookrightarrow V):=(L_0\hookrightarrow\cdots\hookrightarrow
                L_i\hookrightarrow L_i\hookrightarrow\cdots\hookrightarrow L_n\hookrightarrow V).
            \end{equation*}
    \end{enumerate}
\end{definition}

The simplicial object $\Gr^{\leq}_{\bullet}(\C)$ allows us to introduce the index map.

\begin{definition}\label{def:CatIndexMap}
    Let $\C$ be an exact category. The \emph{categorical index map} is the span of simplicial maps
    \begin{equation}\label{catindex}
        \elTate(\C)\longleftarrow\Gr_\bullet^\le(\C)\to^\Index S_\bullet(\C),
    \end{equation}
    where the left-facing arrow is given on $n$-simplices by the assignment
    \begin{equation*}
        (L_0\hookrightarrow\cdots\hookrightarrow L_n\hookrightarrow V)\mapsto V,
    \end{equation*}
    and $\Index$ is given on $n$-simplices by the assignment
    \begin{equation*}
        (L_0\hookrightarrow\cdots\hookrightarrow L_n\hookrightarrow V)\mapsto(L_1/L_0\hookrightarrow\cdots\hookrightarrow
        L_n/L_0).
    \end{equation*}
\end{definition}

\subsection{The $K$-Theoretic Index Map}
In this section, we explain how the categorical index map determines an index map in $K$-theory.

\begin{proposition}\label{prop:gr}
    Let $\C$ be an idempotent complete exact category. Then the map from $\Gr_\bullet^\le(\C)$ to $\elTate(\C)$ of
    \eqref{catindex} induces an equivalence
    \begin{equation*}
        |\Gr_\bullet^\le(\C)^\times|\to^\simeq|\elTate(\C)^\times|.
    \end{equation*}
\end{proposition}
The core of the proof is the fact that the fibers of the above map are classifying spaces of filtered partially ordered sets, and hence contractible. In more detail:
\begin{proof}[Proof of Proposition \ref{prop:gr}]
    Geometric realizations are homotopy colimits. Similarly, the groupoid $\Gr^{\leq}_n(\C)^\times$ is
    the Grothendieck construction of the set-valued functor $$\Gr_n^{\leq}(-)\colon \elTate(\C)^{\times} \to \Set,$$ i.e. it is
    a homotopy colimit in the category of groupoids.

    Commuting the two homotopy colimits we obtain the following equivalence
    \begin{equation*}
        |\Gr_{\bullet}^{\leq}(\C)^\times| \simeq \hocolim_{V \in \elTate(\C)^{\times}} |\Gr_\bullet^{\leq}(V)|.
    \end{equation*}
    For a fixed Tate object, the geometric realization $|\Gr_{\bullet}^{\leq}(V)|$ is the classifying space of the
    category corresponding to the partially ordered set $\Gr(V)$. Because the partially ordered set $\Gr(V)$ is directed (Theorem    \ref{thm:sato_filtered}(c)), its classifying space is contractible. This implies that
    \begin{equation*}
        |\Gr^{\leq}_{\bullet}(\C)^\times| \simeq \hocolim_{\elTate(\C)^{\times}} \{\star\} |\simeq \elTate(\C)^{\times}|.
    \end{equation*}
\end{proof}

\begin{corollary}\label{cor:index}
    Let $\C$ be idempotent complete. The categorical index map determines a map
    \begin{equation*}
        \Index\colon\elTate(\C)^\times\to|S_\bullet(\C)^\times| \cong BK_{\C}.
    \end{equation*}
\end{corollary}

Moreover, using the canonical equivalence $S_k(\elTate(\C))\simeq\elTate(S_k(\C))$ \cite[Prop. 5.13]{Braunling:2014fk}, we also
obtain the following.
\begin{corollary}\label{cor:enhanced}
    Let $\C$ be idempotent complete. The categorical index map determines a map of infinite loop spaces
    \begin{equation}\label{infiniteloop}
        B\Index\colon |S_\bullet(\elTate(\C))^\times|\to|S_\bullet S_\bullet(\C)^\times|
    \end{equation}
    which fits into a commuting triangle
    \begin{equation}\label{connKtriang}
        \xymatrix{
            \elTate(\C)^\times \ar[d] \ar[rr]^{\Index} && \Omega|S_\bullet S_\bullet(\C)^\times|       \\
               \Omega|S_\bullet\elTate(\C)^\times| \ar[urr]_{\qquad \Omega B\Index}                     }
    \end{equation}
\end{corollary}
\begin{proof}
    Recall from \cite[p. 329]{MR0802796} that the infinite loop space structures on $|S_\bullet\elTate(\C)^\times|$ and
    $|S_\bullet S_\bullet(\C)^\times|$ are given by the chain of equivalences
    \begin{align*}
        |S_\bullet^{\times n}(\elTate(\C))^\times|&\to^\simeq\Omega|S_\bullet^{\times n+1}(\elTate(\C))^{\times}|,\intertext{and}
        |S_\bullet^{\times n+1}(\C)^\times|&\to^\simeq\Omega|S_\bullet^{\times n+2}(\C)^{\times}|,
    \end{align*}
    which are determined by the inclusions of 1-simplices
    \begin{align*}
        \Delta^1\times S_\bullet^{\times n}(\elTate(\C))^\times&\to S_\bullet^{\times n+1}(\elTate(\C))^\times,\intertext{and}
        \Delta^1\times S_\bullet^{\times n+1}(\C)^\times&\to S_\bullet^{\times n+2}(\C)^\times.
    \end{align*}
    For each $n$, we have a map
    \begin{equation*}
        |S_\bullet^{\times n}(\elTate(\C))^\times|\to^\simeq |\elTate(S_\bullet^{\times
        n}(\C))^\times|\to^{\simeq}|\Gr^\le_\bullet(S_\bullet^{\times n}(\C))^\times|\to|S_\bullet S_\bullet^{\times n}(\C)^\times|.
    \end{equation*}
    By inspection, these maps fit into a commuting square
    \begin{equation*}
        \xymatrix{
          S_\bullet^{\times n}(\elTate(\C))^\times \ar[d] \ar[rr]^{\Index} && |S_\bullet S_\bullet^{\times n}(\C)^\times| \ar[d] \\
          \Omega|S_\bullet^{\times n+1}(\elTate(\C))^\times| \ar[rr]^{\Omega B\Index} && \Omega|S_\bullet S_\bullet^{\times
          n+1}(\C)^\times|   }
    \end{equation*}
    for each $n$.  For $n=0$, this square gives the triangle \eqref{connKtriang}. Taken together for all $n$, these squares show
    that \eqref{infiniteloop} is a map of infinite loop spaces.
\end{proof}

\subsection{The Index Map is an Equivalence}
Our goal in this section is to prove the following theorem.
\begin{theorem}\label{thm:deloop}
    Let $\C$ be an idempotent complete exact category. Then the $K$-theoretic index map is an equivalence.
\end{theorem}

\begin{rmk}
    {In Section \ref{sec:indexboundary}, we show (independently of Theorem \ref{thm:deloop}) that the $K$-theoretic index map is equivalent to $-1$ times Saito's equivalence. In light of this result, the theorem above is a direct consequence Saito's delooping result \cite[Theorem 1.2]{Saito:2012uq}.}
    { We deduce the present theorem directly from the definition of the index map, and Waldhausen's Additivity Theorem \cite{MR0802796}.}  In Section \ref{sec:aj}, we explain how the present theorem can be understood as the analogue for algebraic $K$-theory of the equivalence, due to Atiyah and J\"{a}nich, between the space of Fredholm operators on a separable complex Hilbert space and the classifying space of topological complex $K$-theory.
\end{rmk}

\begin{proof}[Proof of Theorem \ref{thm:deloop}]
    It suffices to prove that, for each $n$, the map
    \begin{equation*}
        \Gr_n^\le(\C)\to^{\Index} S_n(\C)
    \end{equation*}
    induces an equivalence
    \begin{equation*}
        |S_\bullet(\Gr_n^\le(\C))^\times|\to^\simeq |S_\bullet S_n(\C)^\times|.
    \end{equation*}
    For $n=0$, observe that $\Gr_0^\le(\C)=\E(\Pro(\C),\elTate(\C),\Ind(\C))$. Therefore, by the Additivity Theorem (Theorem
    \ref{thm:wald_add2}), we have an equivalence
    \begin{equation*}
        |S_\bullet(\Gr_0^\le(\C))^\times| \to^\simeq |S_\bullet(\Pro(\C))^\times|\times |S_\bullet(\Ind(\C))^\times|.
    \end{equation*}
    The Eilenberg swindle shows that the right hand side is contractible. We conclude that the map
    \begin{equation*}
        |S_\bullet(\Gr_0^\le(\C))^\times|\to \ast=|S_\bullet S_0(\C)^\times|
    \end{equation*}
    is an equivalence.

    Next, consider the functor
    \begin{equation*}
        \Gr_n^\le(\C)\to\E(\Pro(\C),\Gr_n^\le(\C),S^\ell_n(\C\subset\Ind(\C)))
    \end{equation*}
    which sends $(L_0\hookrightarrow\cdots\hookrightarrow L_n\hookrightarrow V)$ to
    \begin{equation*}
        \xymatrix{
             L_0 \ar@{^{(}->}[r] \ar@{^{(}->}[d] & L_0 \ar@{^{(}->}[r] \ar@{^{(}->}[d] & \cdots \ar@{^{(}->}[r] & L_0
             \ar@{^{(}->}[r] \ar@{^{(}->}[d] & L_0 \ar@{^{(}->}[d] \\
             L_0 \ar@{^{(}->}[r] \ar@{>>}[d] & L_1 \ar@{^{(}->}[r] \ar@{>>}[d] & \cdots \ar@{^{(}->}[r] & L_n \ar@{^{(}->}[r]
             \ar@{>>}[d] & V \ar@{>>}[d]\\
             0 \ar@{^{(}->}[r] & L_1/L_0 \ar@{^{(}->}[r] & \cdots \ar@{^{(}->}[r] & L_n/L_0 \ar@{^{(}->}[r] & V/L_0.
         }
    \end{equation*}
    A quick check shows that this is an equivalence of exact categories. By the Additivity Theorem (Theorem \ref{thm:wald_add1})
    and by Proposition \ref{prop:Slprops}, we conclude that the map
    \begin{align*}
        \Gr_n^\le(\C)&\to\Pro(\C)\times S_n(\C)\times\Ind(\C)\\
        (L_0\hookrightarrow\cdots\hookrightarrow L_n\hookrightarrow V)&\mapsto(L_0,L_1/L_0\hookrightarrow\cdots\hookrightarrow
        L_n/L_0,V/L_n)
    \end{align*}
    induces an equivalence
    \begin{equation*}
        |S_\bullet(\Gr_n^\le(\C))^\times|\to^\simeq |S_\bullet\Pro(\C)^\times| \times |S_\bullet(S_n(\C))^\times| \times
        |S_\bullet(\Ind(\C))^\times|.
    \end{equation*}
    By the Eilenberg swindle, the projection
    \begin{equation*}
        |S_\bullet\Pro(\C)^\times|\times |S_\bullet (S_n(\C))^\times| \times |S_\bullet(\Ind(\C))^\times|\to |S_\bullet
        (S_n(\C))^\times|
    \end{equation*}
    is an equivalence. The map
    \begin{equation*}
        |S_\bullet(\Gr_n^\le(\C))^\times|\to^{\Index} |S_\bullet (S_n(\C))^\times|
    \end{equation*}
    is the composite of the above two maps, and is therefore an equivalence.
\end{proof}

\begin{corollary}\label{cor:tatek=k'}
    Let $\C$ be an idempotent complete exact category, and let $\kappa\le\kappa'$ be a pair of infinite cardinals. Then the exact functor $\elTatek(\C)\to \elTatekp(\C)$ induces an equivalence in $K$-theory $K_{\elTatek(\C)}\to^\simeq K_{\elTatekp(\C)}$.
\end{corollary}
\begin{proof}
    We have a commuting diagram
    \begin{equation*}
        \xymatrix{
            \elTatek(\C) \ar[d] & \Gr_{\kappa,\bullet}^\le(\C) \ar[l] \ar[d] \ar[r]^{\Index} & S_\bullet(\C) \ar[d]^1 \\
            \elTatekp(\C) & \Gr_{\kappa',\bullet}^\le(\C) \ar[l] \ar[r]^{\Index} & S_\bullet(\C)
        }
    \end{equation*}
    After applying $|S_\bullet(-)^\times|$, all of the horizontal arrows become equivalences (by Proposition \ref{prop:gr} and Theorem \ref{thm:deloop}). The corollary now follows from the 2 of 3 property for equivalences.
\end{proof}

\subsection{The Index Map as a Boundary Map}\label{sec:indexboundary}
\begin{theorem}\label{thm:loc}
    Let $\C$ be an idempotent complete exact category. Let
    \begin{equation*}
        \partial\colon\Omega K_{\Ind(\C)/\C}\to K_{\C}
    \end{equation*}
    be the boundary map in the $K$-theory localization sequence associated to the left s-filtering embedding $\C\subset\Ind(\C)$. Then the $K$-theoretic index map fits into a canonical homotopy commuting triangle
    \begin{equation}\label{indexloc}
        \xymatrix{
            \Omega K_{\elTate(\C)} \ar[d]_q \ar[rr]^{\Omega^2 B\Index} && K_{\C}       \\
            \Omega K_{\Ind(\C)/\C} \ar[urr]_{-\partial}                     }.
    \end{equation}
    where the left vertical map is given by the functor $\elTate(\C)\to^q\Ind(\C)/\C$ which sends an elementary Tate
    object $V$ to $V/L$ for any choice of lattice $L\hookrightarrow V$.
\end{theorem}

The proof of Saito's Delooping Theorem \cite{Saito:2012uq} implies the following.
\begin{corollary}\label{cor:minuseins}
    Let $\C$ be idempotent complete. The index map is canonically equivalent to $-1$ times Saito's delooping
    \begin{equation*}
        \Omega K_{\elTate(\C)}\to^\simeq K_{\C}.
    \end{equation*}
\end{corollary}

\begin{proof}[Proof of Theorem \ref{thm:loc}]
    The assignment
    \begin{equation*}
        (L_0\hookrightarrow\cdots\hookrightarrow L_n\hookrightarrow V)\mapsto(L_1/L_0\hookrightarrow\cdots
        L_n/L_0;L_1/L_0\hookrightarrow\cdots\hookrightarrow L_n/L_0\hookrightarrow V/L_0)
    \end{equation*}
    determines a map of simplicial diagrams of categories
    \begin{equation*}
        \Gr_\bullet^\le(\C)\to S^\ell_\bullet(\C\subset\Ind(\C))
    \end{equation*}
    which fits into a 2-commuting triangle
    \begin{equation*}
        \xymatrix{
          \Gr_\bullet^\le(\C) \ar[d] \ar[rr]^{\Index} && S_\bullet(\C)        \\
          S^\ell_\bullet(\C\subset\Ind(\C)) \ar[urr]_{\delta}                     }.
    \end{equation*}
    We also have a 2-commuting square
    \begin{equation*}
        \xymatrix{
          \elTate(\C) \ar[d]  && \Gr_\bullet^\le(\C) \ar[ll] \ar[d] \\
          \Ind(\C)/\C  && S^\ell_\bullet(\C\subset\Ind(\C)) \ar[ll]   }
    \end{equation*}
    where the bottom horizontal map is the restriction to 1-simplices of the equivalence appearing in the proof of Proposition
    \ref{prop:Schlichting1}, and where the left vertical map is the functor described above.

    Applying the $S$-construction, we obtain a commuting diagram of spaces
    \begin{equation*}
        \xymatrix{
            |S_\bullet(\elTate(\C))^\times| \ar[d] && |S_\bullet(\Gr_\bullet^\le(\C))^\times| \ar[ll]_\simeq \ar[d]
            \ar[rr]^{B\Index} && |S_\bullet S_\bullet(\C)^\times|\\
            |S_\bullet(\Ind(\C)/\C)^\times| && |S_\bullet S^\ell_\bullet(\C\subset\Ind(\C))^\times| \ar[ll]^\simeq \ar[urr]
            }.
    \end{equation*}
    By Proposition \ref{prop:Schlichting}, after inverting the left-facing equivalences and taking the double loop spaces, we obtain a contractible space of homotopy commuting triangles of the form \eqref{indexloc}.
\end{proof}

\subsection{A Combinatorial Model of the Index Map}\label{sub:indexaut}
In this section, we introduce convenient simplicial models of $|\Gr_\bullet^\le(\C)^\times|$ and $|S_\bullet(\C)^\times|$ which allow us to construct the map
\begin{equation*}
    |\elTate(\C)^\times|\to |S_\bullet(\C)^\times|
\end{equation*}
as an explicit simplicial map from the nerve of $\elTate(\C)^\times$.

\begin{rmk}
    We pause for a moment to explain the data exhibited by such a map.
    \begin{enumerate}
        \item From the perspective which we adopt in this section, the combinatorial model of the index map can be understood as a universal computation of indices, symbols, and higher torsion invariants of automorphisms of elementary Tate objects. In order to define the computation, we require a sequence of auxiliary choices. Theorem \ref{thm:sato_filtered} ensures that the data required for these choices exist, while the framework of simplicial homotopy theory ensures that the end-result is independent of the choices.\footnote{Though we do not pursue this here, we also expect that, given two sequences of such auxiliary choices, one can directly construct a homotopy between the resulting simplicial maps by a sequence of similar choices.} We conclude this section with a sample computation which explains the name of the index map.
        \item From another perspective, this simplicial map encodes an $A_\infty$-map
            \begin{equation*}
                \Aut(V)\to K_{\C}
            \end{equation*}
            for all elementary Tate objects $V$. More precisely, recall (e.g. \cite[Chapter V]{GoerssJardine:09}) that there exists a model structure, essentially due to Kan, on the category $\sSetr$ of \emph{reduced} simplicial sets, i.e. simplicial sets having a unique vertex. There is also a model structure, essentially due to Moore, on the the category $\sGrp$ of simplicial groups, and a Quillen equivalence
            \begin{equation*}
                G \colon \sSetr\rightleftarrows \sGrp \colon \overline{W},
            \end{equation*}
            with $G\dashv \overline{W}$. At the level of the underlying $\infty$-categories, this equivalence corresponds to the adjoint equivalence $\Omega\dashv B_\bullet$ between the $\infty$-category of group-like $A_\infty$-spaces and the $\infty$-category of pointed connected spaces. Both the nerve $N_\bullet\Aut(V)$ of the group $\Aut(V)$ and the model we use of $BK_{\C}$ are reduced simplicial sets, so the construction falls within this classical framework.

            However, there exists a different, and, for many purposes, more natural approach to $E_1$-objects in an $\infty$-category, essentially due to Segal. In \cite{BGW:Segal}, we develop a formalism which allows us to produce an efficient Segal-style model for this $A_\infty$-map in the $\infty$-category of spaces.
        \item From a third perspective, which we develop in Section \ref{sub:torsor}, this simplicial map can be understood as specifying the data of a $K_{\C}$-torsor $\mathcal{T}\to\elTate(\C)^\times$, or as specifying, for each elementary Tate object $V$, a $K_{\C}$-torsor $\mathcal{T}|_V$ with a coherent action of $\Aut(V)$.
    \end{enumerate}
\end{rmk}

\begin{notation}
    Throughout this section, $I\subset[n]$ will denote a \emph{non-empty} sub-set.
\end{notation}
Given a groupoid $\mathcal{G}$ and a functor $F_\bullet\colon\mathcal{G}\to\sSet$, denote by $\int_{\mathcal{G}} F_\bullet$ the Grothendieck construction of $F_\bullet$. Explicitly, $\int_{\mathcal{G}} F_\bullet$ is the simplicial diagram of categories whose category of $n$-simplices is the usual Grothendieck construction of $F_n$.

Given an elementary Tate object $V$, denote by $\Gr^\le_\bullet(V)$ the nerve of the partially ordered set $\Gr(V)$. As we observed in the proof of Proposition \ref{prop:gr}, the assignment $V\mapsto \Gr^\le_\bullet(V)$ defines a functor
\begin{equation*}
    \Gr^\le_\bullet\colon\elTate(\C)^\times\to\sSet
\end{equation*}
with $\int_{\elTate(\C)^\times} \Gr^\le_\bullet=\Gr_\bullet^\le(\C)^\times$. Starting with this observation, we now introduce several constructions which allow us to define a simplicial model for the inverse of the equivalence $|\Gr_\bullet^\le(\C)^\times|\to^\simeq |\elTate(\C)^\times|$.

\subsubsection{Subdivision and Kan's \texorpdfstring{$\Ex$}{Ex}}
\begin{definition}
    The \emph{subdivision} of the linearly ordered set $[n]$, denoted $\sd([n])$, is the partially ordered set consisting of all non-empty sub-sets $I\subset[n]$, ordered by inclusion.
\end{definition}

By taking the nerve of $\sd([n])$, we obtain a functor
\begin{equation*}
    \Ddelta\to\sSet
\end{equation*}
The left Kan extension of this functor along the Yoneda embedding gives a functor
\begin{equation*}
    \sd\colon\sSet\to\sSet.
\end{equation*}

\begin{example}
    The simplicial set $\sd\Delta^1$ consists of two 1-simplices $x_0$ and $x_1$ glued at their ends
    \begin{equation*}
        \xymatrix{
            \bullet \ar[rr]^{x_0} && \bullet && \ar[ll]_{x_1} \bullet
        }
    \end{equation*}
\end{example}

\begin{definition}
    Let $X_\bullet$ be a simplicial set. Define $\Ex(X)_\bullet$ to be the simplicial set whose $n$-simplices are given by
    \begin{equation*}
        \Ex(X)_n:=\hom_{\sSet}(\sd\Delta^n,X_\bullet).
    \end{equation*}
\end{definition}

\begin{example}
    The example above shows that a 1-simplex of $\Ex(X)$ consists of two 1-simplices $x_0$ and $x_1$ of $X$ glued at their ends.
\end{example}

The assignment $I\mapsto\max(I)$ defines a natural map of partially ordered sets
\begin{equation*}
    \sd([n])\to [n].
\end{equation*}
This extends to a natural transformation $\sd(-)\to 1_{(-)}$, which, in turn, defines a natural transformation $1_{(-)}\to \Ex(-)$. The following is one of the foundational results of simplicial homotopy theory.

\begin{lemma}[Kan](\cite{MR0090047}, cf. \cite[Theorem III.4.6]{GoerssJardine:09})\label{lemma:exequiv}
    Let $X_\bullet$ be a simplicial set. The map $X_\bullet\to\Ex(X)_\bullet$ is a weak equivalence.
\end{lemma}

When $X_\bullet$ is the nerve of a poset $P$, $\Ex(X)_\bullet$ admits a particularly simple description.
\begin{example}\label{example:exofposet}
    Let $P$ be a partially ordered set. Denote by $P_\bullet$ its nerve. Then
    \begin{equation*}
        \Ex(P)_n\cong\{(x_I)_{I\subset [n]}~|~x_I\in P \text{ for all $I$, and } x_I<x_J \text{ for }I\subset J\}.
    \end{equation*}
    The face $i^{th}$ face of a simplex $(x_I)_{I\subset [n]}$ is given by the functor
    \begin{equation*}
        J\mapsto x_{d^i(J)}
    \end{equation*}
    Conversely, the $i^{th}$ degeneracy of a simplex $(x_I)_{I\subset [n]}$ is given by the functor
    \begin{equation*}
        J\mapsto x_{s^i(J)}.
    \end{equation*}
\end{example}

Under the inclusion $\Set\into\Cat$, we obtain a functor
\begin{equation*}
    \sd^h\colon\sSet\to^{\sd}\sSet\into\sCat.
\end{equation*}

\begin{definition}
    Let $X_\bullet\colon\Ddelta^{\op}\to\Cat$ be a simplicial category. Define $\Exh(X)_\bullet$ to be the simplicial category with
    \begin{equation*}
        \Exh(X)_n:=\Fun_{\sCat}(\sd^h\Delta^n,X_\bullet).
    \end{equation*}
\end{definition}

From the definition, we have $\ob\Exh(X)_\bullet\cong \Ex(\ob X_\bullet)$. Morphisms in $\Exh(X)_\bullet$ are morphisms of diagrams. This forms the basis of the following lemma.
\begin{lemma}\label{lemma:exofGroth}
    Let $\mathcal{G}$ be a groupoid, and let $F_\bullet\colon \mathcal{G}\to\sSet$ be a functor taking values in simplicial sets. Then
    \begin{equation*}
        \Exh\int_{\mathcal{G}} F_\bullet=\int_{\mathcal{G}} \Ex F_\bullet.
    \end{equation*}
\end{lemma}
\begin{proof}
    From the definition, objects of the category $\Exh\int_{\mathcal{G}} F_n$ are pairs $(V,L)$, where $V\in\mathcal{G}$ and $L\in \Ex(F(V))_n$. A morphism $(V_0,L_0)\to^g (V_1,L_1)$ in $\Exh\int_{\mathcal{G}} F_n$ consists of a morphism $V_0\to^g V_1$ in $\mathcal{G}$ such that $F(g)(L_0)=L_1$. We see that $\Exh\int_{\mathcal{G}}F_n$ is, by definition, the category $\int_{\mathcal{G}} \Ex(F)_n$. A similar exercise shows the equality of face and degeneracy maps.
\end{proof}

\begin{lemma}\label{lemma:exhequiv}
    Let $X_\bullet$ be a simplicial diagram of categories. Then the map
    \begin{equation*}
        X_\bullet\to \Exh X_\bullet
    \end{equation*}
    induces a weak equivalence
    \begin{equation*}
        |X|\to^\simeq |\Exh X|.
    \end{equation*}
\end{lemma}
\begin{proof}
    Given a simplicial diagram of categories $X_\bullet$, let $X_{\bullet,\bullet}$ be the bisimplicial set obtained by taking the nerve (in the vertical direction) of the categories $X_n$. Unpacking the definition, we see that $\Exh X_{\bullet,\bullet}$ is the bisimplicial set with
    \begin{equation*}
        (\Exh X)_{\bullet,n}=\Ex X_{\bullet,n}.
    \end{equation*}
    Similarly, the map
    \begin{equation*}
        X_{\bullet,\bullet}\to \Exh X_{\bullet,\bullet}
    \end{equation*}
    is the map of  bisimplicial sets given on horizontal $n$-simplices by
    \begin{equation*}
        X_{\bullet,n}\to \Ex X_{\bullet,n}.
    \end{equation*}
    By Lemma \ref{lemma:exequiv}, this is a weak equivalence for all $n$.
\end{proof}

\subsubsection{The Diagonal of the Grothendieck Construction}
Let $\mathcal{G}$ be a groupoid, and let $F_\bullet\colon \mathcal{G}\to\sSet$ be a functor taking values in simplicial sets. Taking the nerves of the groupoids $\int_{\mathcal{G}} F_\bullet$ (in the vertical direction), we obtain a bisimplicial set $\int_{\mathcal{G}}F_{\bullet,\bullet}$.

Concretely, $(n,m)$-simplices of $\int_{\mathcal{G}}F_{\bullet,\bullet}$ consist of a string of isomorphisms in $\mathcal{G}$
\begin{equation*}
    x_0\to^{g_1}\cdots\to^{g_m} x_m
\end{equation*}
along with elements $y_i\in F_n(x_i)$ such that $F(g_i)(y_i)=y_{i+1}$. Because all the $g_i$ are isomorphisms, we see that $(n,m)$-simplices of $\int_{\mathcal{G}}F_{\bullet,\bullet}$ are equivalent to tuples
\begin{equation*}
    (x_0\to^{g_1}\cdots\to^{g_m}x_m,y)\in N_m\mathcal{G}\times F_n(x_m).
\end{equation*}
Applying the diagonal functor, we see that $n$-simplices of $d(\int_{\mathcal{G}}F)_\bullet$ consist of tuples
\begin{equation*}
    (x_0\to^{g_1}\cdots\to^{g_n} x_n,y)\in N_n\mathcal{G}\times F_n(x_n).
\end{equation*}
The degeneracy $s_i$ is the product of the $i^{th}$ degeneracy maps in $F_\bullet(x_n)$ and $N_\bullet\mathcal{G}$. For $i<n$, the face map $d_i$ is given by the product of the $i^{th}$ face maps in $F_\bullet(x_n)$ and $N_\bullet\mathcal{G}$, while the face map $d_n$ is given by
\begin{equation*}
    d_n(x_0\to^{g_1}\cdots\to^{g_n} x_n,y):=(x_0\to^{g_1}\cdots\to^{g_{n-1}} x_{n-1},F(g_n)^{-1}(y)).
\end{equation*}

In the case $\Exh\Gr_\bullet^\le(\C)^\times=\int_{\elTate(\C)^\times} \Ex\Gr^\le_\bullet$, the previous description combines with Lemma \ref{lemma:exofGroth} and Example \ref{example:exofposet}, to give the following.
\begin{lemma}\label{lemma:dgr}
    Let $\C$ be an exact category. Then
    \begin{equation*}
        |\Gr^\le_\bullet(\C)^\times|\simeq|d(\Exh \Gr_\bullet^\le(\C)^\times)|.
    \end{equation*}
    Further, the $n$-simplices of the simplicial set $d(\Exh\Gr^\le_\bullet(\C)^\times)_\bullet$ consist of tuples
    \begin{equation*}
        (V_0\to^{g_1}\cdots\to^{g_n}V_n,\{L_I\}_{I\subset[n]})\in N_n\elTate(\C)^\times\times \Ex\Gr^\le_n(V_n).
    \end{equation*}
    In this description, the degeneracy $s_i$ is the product of the $i^{th}$ degeneracy maps in $N_\bullet\elTate(\C)^\times$ and $\Ex\Gr^\le_\bullet(V_n)$. For $i<n$, the face map $d_i$ is given by the product of the $i^{th}$ face maps in $N_\bullet\elTate(\C)^\times$ and $\Ex\Gr^\le_\bullet(V_n)$, while the face map $d_n$ is given by
    \begin{equation*}
        d_n(V_0\to^{g_1}\cdots\to^{g_n} V_n,\{L_I\}_{I\subset[n]}):=(V_0\to^{g_1}\cdots\to^{g_{n-1}} V_{n-1},\{g_n^{-1}L_{d^n(J)}\}_{J\subset[n-1]}).
    \end{equation*}
\end{lemma}

\subsubsection{A Combinatorial Model of the Index Map}
\begin{theorem}\label{thm:grsec}
    Let $\C$ be an idempotent complete exact category. A section of the map
    \begin{equation}\label{exgraug}
        d(\Exh\Gr^\le_\bullet(\C)^\times)_\bullet\to N_\bullet\elTate(\C)^\times.
    \end{equation}
    is constructed according to the following induction.
    \begin{enumerate}
        \item For each $V\in\elTate(\C)^\times$, choose a lattice $L_{0,[0]}\into V$.
        \item For the inductive step, suppose that for each $k<n$, for each non-degenerate simplex
            \begin{equation*}
                \overline{g}=(V_0\to^{g_1}\cdots\to^{g_k} V_k)\in N_k\elTate(\C)^\times,
            \end{equation*}
            and for each non-empty subset $I\subset [k]$, we have specified a collection of lattices $L_{k,I}(\overline{g})\into V_k$ satisfying:
                \begin{enumerate}
                    \item if $I\subset d^i([k-1])\subset[k]$ for $i<k$, then
                        \begin{equation*}
                            L_{k,I}(\overline{g})=L_{k-1,s^i(I)}(d_i\overline{g}),
                        \end{equation*}
                    \item if $I\subset d^k([k-1])\subset[k]$, then
                        \begin{equation*}
                            L_{k,I}(\overline{g})=g_k L_{k-1,s^{k-1}(I)}(d_k\overline{g}),
                        \end{equation*}
                    \item and, for all $I\subset J\subset [k]$, $L_{k,I}(\overline{g})$ is a sub-lattice of $L_{k,J}(\overline{g})$.
                \end{enumerate}
                Then, let
                \begin{equation*}
                    \overline{g}=(V_0\to^{g_1}\cdots\to^{g_n}V_n)\in N_n\elTate(\C)^\times
                \end{equation*}
                be a non-degenerate simplex, let $I\subset[n]$ be a proper, non-empty subset, and let $i$ be any number such that $I\subset d^i([n-1])\subset [n]$. Define
                \begin{equation*}
                    L_{n,I}(\overline{g}):=\left\{
                        \begin{array}{lr}
                            L_{n-1,s^i(I)}(d_i\overline{g}) & \text{ if } i<n\\
                            \\
                            g_n L_{n-1,s^{n-1}(I)}(d_n\overline{g}) & \text{ if } i=n.
                        \end{array}\right.
                \end{equation*}
                Then $L_{n,I}(\overline{g})$ is a well-defined lattice of $V_n$, independent of the choice of $i$, and we complete the inductive step by choosing a lattice  $L_{n,[n]}(\overline{g})\into V_n$ which contains $L_{n,I}(\overline{g})$ as a sub-lattice for all proper, non-empty subsets $I\subset[n]$.
    \end{enumerate}
    The induction constructs a collection of lattices $\{L_{n,I}(\overline{g})\into V_n\}_{I\subset[n]}$ for every $n$ and every non-degenerate simplex $\overline{g}\in N_n\elTate(\C)^\times$. These families are such that the assignment
    \begin{equation}\label{explicitindexmap}
        \overline{g}\mapsto(\overline{g},\{L_{n,I}(\overline{g})\into V_n\}_{I\subset[n]})
    \end{equation}
    defines a right inverse
    \begin{equation*}
        N_\bullet\elTate(\C)^\times\to^{\mathcal{L}} d(\Exh\Gr_\bullet^\le(\C)^\times)_\bullet
    \end{equation*}
    to the map \eqref{exgraug}.
\end{theorem}
\begin{rmk}
    The theorem is properly understood as a construction in the setting of groups acting on directed posets.  To wit, for any directed poset $P$ and any group $G$ acting on $P$, the induction of the theorem extends, mutatis mutandis, to construct a section of the map $\Ex(P)//G\to N_\bullet G$.
\end{rmk}
\begin{proof}
    For ease of notation, we will leave the elementary Tate object $V_n$ implicit in the course of the proof, e.g. we write \eqref{explicitindexmap} as
    \begin{equation*}
        \overline{g}\mapsto(\overline{g},\{L_{n,I}(\overline{g})\}_{I\subset[n]}).
    \end{equation*}
    For the proof, we need to establish the following claims:
    \begin{enumerate}
        \item That for a proper, non-empty sub-set $I\subset[n]$, $L_{n,I}(\overline{g})$ is well-defined.
        \item That we can make the choices of $L_{n,[n]}(\overline{g})$ required for the induction.
        \item That, for each non-degenerate $n$-simplex $\overline{g}\in N_n\elTate(\C)^\times$, the collection $\{L_{n,I}(\overline{g})\}_{I\subset[n]}$ of the inductive step satisfies the inductive hypothesis for the next stage.
        \item That the assignment \eqref{explicitindexmap} defines a simplicial map.
    \end{enumerate}
    For the first claim, suppose that $i<j<n$ are in $[n]\setminus I$ (i.e. $I\subset d^i([n-1])\cap d^j([n-1])\subset[n]$). We claim that
    \begin{equation*}
        L_{n-1,s^i(I)}(d_i\overline{g})=L_{n-1,s^j(I)}(d_j\overline{g}).
    \end{equation*}
    By assumption, $s^j(I)\subset d^i([n-2])$, while $s^i(I)\subset d^{j-1}([n-2])$. Therefore, by the inductive hypothesis and by the simplicial identities, we have
    \begin{align*}
        L_{n-1,s^j(I)}(d_j\overline{g})&=L_{n-2,s^is^j(I)}(d_id_j\overline{g})\\
        &=L_{n-2,s^{j-1}s^i(I)}(d_{j-1}d_i\overline{g})\\
        &=L_{n-1,s^i(I)}(d_i\overline{g})
    \end{align*}
    as asserted. It remains to show the case where $j=n$. For this, we first suppose that $i<n-1$. We claim that
    \begin{equation*}
        L_{n-1,s^i(I)}(d_i\overline{g})=g_nL_{n-1,s^{n-1}(I)}(d_n\overline{g}).
    \end{equation*}
    By assumption, $s^{n-1}(I)\subset d^i([n-2])$, while $s^i(I)\subset d^{n-1}([n-2])$. Therefore, by the inductive hypothesis and by the simplicial identities, we have
    \begin{align*}
        g_n L_{n-1,s^{n-1}(I)}(d_n\overline{g})&=g_n L_{n-2,s^is^{n-1}(I)}(d_id_n\overline{g})\\
        &=g_n L_{n-2,s^{n-2}s^i(I)}(d_{n-1}d_i\overline{g})\\
        &=L_{n-1,s^i(I)}(d_i\overline{g})
    \end{align*}
    as asserted.

    Finally, suppose that $i=n-1$. Note that, because the maps $s^{n-2}$ and $s^{n-1}$ are equal on $[n]\setminus\{n-1,n\}$, our assumption on $I$ implies that $s^{n-2}(I)=s^{n-1}(I)\subset [n-1]$. Then we claim that
    \begin{equation*}
        L_{n-1,s^{n-1}(I)}(d_{n-1}\overline{g})=g_nL_{n-1,s^{n-1}(I)}(d_n\overline{g}).
    \end{equation*}
    By assumption, $s^{n-1}(I)$ is contained in $d^{n-1}([n-2])\subset[n-1]$. Therefore, by the inductive hypothesis, by the simplicial identities, and by the equality $s^{n-2}(I)=s^{n-1}(I)$, we have
    \begin{align*}
        g_n L_{n-1,s^{n-1}(I)}(d_n\overline{g})&=g_{n-1}g_n L_{n-2,s^{n-2}s^{n-1}(I)}(d_{n-1}d_n\overline{g})\\
        &=g_{n-1}g_n L_{n-2,s^{n-2}s^{n-2}(I)}(d_{n-1}d_{n-1}\overline{g})\\
        &=L_{n-1,s^{n-2}(I)}(d_{n-1}\overline{g})\\
        &=L_{n-1,s^{n-1}(I)}(d_{n-1}\overline{g}).
    \end{align*}
    We have established the first claim.

    The second claim follows, by Theorem \ref{thm:sato_filtered}, from our assumption that $\C$ is idempotent complete. Indeed, for every elementary Tate object $V$, a lattice $L_{0,[0]}\into V$ exists. For the induction step, the lattices $L_{n,I}(\overline{g})\into V_n$ are given, and it remains to choose a lattice $L_{n,[n]}(\overline{g})\into V_n$ containing all of these as proper sub-lattices. Such a lattice exists by Theorem \ref{thm:sato_filtered} (there are finitely many non-empty sub-sets of $[n]$).

    We now turn to the third claim and verify that for each non-degenerate simplex $\overline{g}\in N_n\elTate(\C)^\times$, the collection $\{L_{n,I}(\overline{g})\}_{I\subset[n]}$ satisfies the inductive hypotheses. Because we have already shown that the $L_{n,I}(\overline{g})$ are well-defined, the first two inductive hypotheses are satisfied by definition. It remains to show that if $I\subset J$, then $L_{n,I}(\overline{g})$ is a sub-lattice of $L_{n,J}(\overline{g})$. For $J=[n]$ this follows by definition. Now suppose we are given a proper, non-empty subset $J\subset [n]$. Suppose there exists $i<n$ such that $J\subset d^i([n-1])\subset[n]$. Then for any $I\subset J$, we have
    \begin{align*}
        L_{n,J}(\overline{g})&=L_{n-1,s^i(J)}(d_i\overline{g}),\intertext{and}
        L_{n,I}(\overline{g})&=L_{n-1,s^i(I)}(d_i\overline{g}).
    \end{align*}
    By the inductive hypothesis, because $s^i(I)\subset s^i(J)$, we see that $L_{n,I}(\overline{g})$ is a sub-lattice of $L_{n,J}(\overline{g})$. It remains to consider $J=d^n([n-1])$. For any $I\subset J$, we have
    \begin{align*}
        L_{n,J}(\overline{g})&=g_n L_{n-1,[n-1]}(d_n\overline{g}),\intertext{and}
        L_{n,I}(\overline{g})&=g_n L_{n-1,s^{n-1}(I)}(d_n\overline{g}).
    \end{align*}
    By the inductive hypothesis $L_{n-1,s^{n-1}(I)}(d_n\overline{g})$ is a sub-lattice of $L_{n-1,[n-1]}(d_n\overline{g})$. We conclude that $L_{n,I}(\overline{g})$ is a sub-lattice of $L_{n,J}(\overline{g})$, as required.

    We conclude the proof by showing that we have indeed defined a simplicial map. A map of simplicial sets is uniquely determined by its value on non-degenerate simplices.\footnote{e.g. this follows from the definition of the $n$-skeleton functors, and the fact that every simplicial set is the union of its $n$-skeleta.} Therefore, it is enough to check that the above assignment respects the face maps. For this, we begin by showing that, for all $i<n$, we have
    \begin{align}\label{indexsimpproof1}
        d_i(\overline{g},\{L_{n,I}(\overline{g})\}_{I\subset [n]})=(d_i\overline{g},\{L_{n-1,J}(d_i\overline{g})\}_{J\subset[n-1]}).
    \end{align}
    From Example \ref{example:exofposet} and Lemma \ref{lemma:dgr}, we see that the left hand side is equal to
    \begin{equation*}
        (d_i\overline{g},\{L_{n,I}(\overline{g})\}_{I\subset d^i([n-1])\subset[n]}),
    \end{equation*}
    where the collection $\{L_{n,I}(\overline{g})\}_{I\subset d^i([n-1])\subset[n]}$ denotes the $(n-1)$-simplex of $\Ex\Gr^\le_\bullet(V)$ which sends $J\subset [n-1]$ to $L_{n,d^i(J)}(\overline{g})$. On the other hand, by the inductive definition and the simplicial identities, we have that
    \begin{align*}
        L_{n,d^i(J)}(\overline{g})&:=L_{n-1,s^id^iJ}(d_i\overline{g})\\
        &=L_{n-1,J}(d_i\overline{g}).
    \end{align*}
    This establishes the equality \eqref{indexsimpproof1}.

    We must also show that
    \begin{align}\label{indexsimpproof2}
        d_n(\overline{g},\{L_{n,I}(\overline{g})\}_{I\subset [n]})=(d_n\overline{g},\{L_{n-1,J}(d_n\overline{g})\}_{J\subset[n-1]}).
    \end{align}
    From Example \ref{example:exofposet} and the description of Lemma \ref{lemma:dgr}, we see that the left hand side is equal to
    \begin{equation*}
        (d_n\overline{g},\{g_n^{-1}L_{n,I}(\overline{g})\}_{I\subset d^n([n-1])\subset[n]}),
    \end{equation*}
    where the collection $\{g_n^{-1}L_{n,I}(\overline{g})\}_{I\subset d^n([n-1])\subset[n]}$ denotes the $(n-1)$-simplex of $\Ex\Gr^\le_\bullet(V)$ which sends $J\subset [n-1]$ to $g_n^{-1}L_{n,d^n(J)}(\overline{g})$. On the other hand, by the inductive definition and the simplicial identities, we have that
    \begin{align*}
        g_n^{-1}L_{n,d^n(J)}(\overline{g})&:=g_n^{-1}g_nL_{n-1,s^{n-1}d^nJ}(d_n\overline{g})\\
        &=L_{n-1,J}(d_n\overline{g}).
    \end{align*}
    This establishes the equality \eqref{indexsimpproof2} and we therefore conclude that the map is simplicial.
\end{proof}

\begin{corollary}\label{cor:explicitindex}
    Let $\C$ be idempotent complete. Let $\mathcal{L}$ be a map as in the previous theorem. Then the geometric realization of the composite
    \begin{align}\label{explicitindexmapOF_COROLLARY}
        N_\bullet\elTate(\C)^\times&\to^{\mathcal{L}} d(\Exh\Gr^\le_\bullet(\C)^\times)_\bullet\to^{d\Exh(\Index)} d(\Exh S_\bullet(\C)^\times)_\bullet
    \end{align}
    is equivalent to the index map.
\end{corollary}
\begin{proof}
    By Lemma \ref{lemma:dgr}, we see that $|\mathcal{L}|$ gives a left inverse to the equivalence
    \begin{equation*}
        |\Gr^\le_\bullet(\C)^\times|\to^\simeq|\elTate(\C)^\times|.
    \end{equation*}
    Similarly, by Lemma \ref{lemma:exhequiv} and the equivalence $|X_{\bullet,\bullet}|\simeq|d(X)_\bullet|$, we see that
    \begin{equation*}
        |d\Exh(\Index)|\simeq|\Index|\colon|\Gr^\le_\bullet(\C)^\times|\to|S_\bullet(\C)^\times|.
    \end{equation*}
    The corollary now follows from the definition of the index map.
\end{proof}

The corollary represents a universal calculation of symbols and higher torsion invariants for automorphisms of Tate objects. As an example of the information this contains, we now compute the value of the map
\begin{equation*}
    B\Aut(V)=\pi_1(|\elTate(\C)^\times|,V)\to^\Index \pi_1(|S_\bullet(\C)^\times|)=K_0(\C).
\end{equation*}
The construction above shows that, once we have chosen a lattice $L\into V$ (denoted $L_{0,[0]}$ above), we have
\begin{equation*}
    \mathcal{L}(e)=(e,\{L\}_{I\subset[1]})
\end{equation*}
where the collection $\{L\}_{I\subset[1]}$ denotes the 1-simplex of $\Ex\Gr_\bullet^\le(V)$ which sends $I\subset[1]$ to $L$. Therefore, $\Index(e)$ consists of two copies of the degenerate 1-simplex joined at their ends. We conclude
\begin{equation*}
    [\Index(e)]=0\in K_0(\C)
\end{equation*}
as expected.

Now let $g\in\Aut(V)$ be a non-trivial automorphism. The construction above shows that, having chosen $L$, to define $\mathcal{L}(g)$, it suffices to choose a lattice $N$ which contains both $L$ and $gL$ as sub-lattices (we denote $N$ by $L_{1,[1]}(g)$ above). The map $\mathcal{L}(g)$ sends $g$ to the loop in $d(\Exh\Gr_\bullet^\le(\C)^\times)_\bullet$ given by
\begin{equation*}
    (g,\xymatrix{gL \ar[r] & N & L \ar[l] })\in N_1\elTate(\C)^\times\times \Ex\Gr_1^\le(V).
\end{equation*}
Applying the categorical index map, we obtain the loop in $d(\Exh S_\bullet(\C)^\times)$ given by
\begin{equation*}
    \xymatrix{
        \bullet \ar[rr]^{N/gL} && \bullet && \bullet \ar[ll]_{N/L}
    }
\end{equation*}
Passing to $\pi_1$, this gives
\begin{equation*}
    [\Index(g)]=[N/gL]-[N/L]\in K_0(\C).
\end{equation*}
Corollary \ref{cor:explicitindex} ensures that this value is independent of our choices (as one can also check directly).

\begin{example}\label{ex:laurent}
    Let $k$ be a field. The ring of Laurent series $k((t))$, has a canonical structure of an elementary Tate vector space over
    $k$. An invertible element $f=\sum_{i\geq n}^{\infty} a_it^i \in k((t))^{\times}$ gives an automorphism of the Tate module
    $k((t))$ which takes the lattice $k[[t]]\subset k((t))$ to the lattice $t^n k[[t]]$. Taking $L_{0,[0]}=k[[t]]$ and $L_{1,[1]}(g)=t^{\min{n,0}}k[[t]]$, we conclude that
    \begin{equation*}
        [\Index(f)]=\left\{
            \begin{array}{rl}
                \left[-k\langle t^{-1},\ldots,t^{-n}\rangle\right]\in{ }K_0(k) & \text{if } n < 0\\
                0\in{ }K_0(k) & \text{if } n = 0\\
                \left[k\langle 1,t,\ldots,t^{n-1}\rangle\right] \in{ }K_0(k) & \text{if } n > 0
            \end{array} \right.
    \end{equation*}
    where $[k\langle\ldots\rangle]$ denotes the class of the $k$-vector space with generators $\langle\ldots\rangle$. In particular, if we identify $\pm n$ with $\pm[k^n]\in K_0(k)$ for $n\in\mathbb{N}$, we have $[\Index(f)]=n\in K_0(k)$. So, in this example, $\pi_0$ of the index map recovers the winding number of a non-vanishing formal Laurent series $f$.
\end{example}

\subsection{Comparison with Index Theory in Topological $K$-Theory}\label{sec:aj}
We now relate the index map to similar constructions defined in the context of index theory for Fredholm operators on Hilbert space. The general analogy is well known, and dates back at least to Sato--Sato \cite{MR730247} and Segal--Wilson \cite{SeW:85}. Our goal here is to elaborate this analogy by adding Propositions \ref{prop:plus} and \ref{prop:index=AJ}.

\subsubsection{The $K$-Theory of Tate Objects as an Analogue of Fredholm Operators}\label{subsub:fredholm}
Let $\Hc$ be a complex separable Hilbert space, e.g. $L^2(S^1;\mathbb{C})$. Recall that a bounded operator
\begin{equation*}
    A\colon\Hc\to\Hc
\end{equation*}
is \emph{Fredholm} if $\dim\ker(A)<\infty$ and $\dim\coker(A)<\infty$. Denote by $\Fred(\Hc)$ the space of Fredholm operators (topologized as a subspace of the space of bounded operators on $\Hc$). The space $\Fred(\Hc)$ is endowed with a \emph{tautological complex}
\begin{equation*}
    \Index^\bullet\to\Fred(\Hc)
\end{equation*}
whose fiber at $A\in\Fred(\Hc)$ is the complex
\begin{equation*}
    \Hc\to^A\Hc.
\end{equation*}

\begin{theorem}[Atiyah \cite{Ati:67}, J\"{a}nich \cite{Jan:65}]
    The complex $\Index^\bullet\to\Fred(\Hc)$ is \emph{perfect}, i.e. its restriction to any compact subspace $X\subset\Fred(\Hc)$ is quasi-isomorphic to a bounded complex of finite-dimensional topological vector bundles.
\end{theorem}

A perfect complex $E^\bullet\to X$ on a space $X$ defines a map
$$X \to {}K_{\mathbb{C}}^{top}$$
from $X$ to the classifying space of topological complex $K$-theory, sending $x \in X$ to $\chi(E|_x). $\footnote{We can see this directly as follows. Denote by $\Perf^{top}$ the stack which assigns to a space $X$ its ($\infty$-)category of topological perfect complexes $\Perf^{top}(X)$. Under the Yoneda embedding, a perfect complex $E^\bullet\to X$ is equivalent to a map $X\to^E\Perf^{top}(-)^\times$. We obtain the map $X\to^E K^{top}_\mathbb{C}$ by composing $X\to^E\Perf^{top}(-)^\times$ with the canonical map $\Perf^{top}(-)^\times\to K_{\Perf^{top}(-)}$, followed by the equivalence $K_{\Perf^{top}(-)}\simeq K^{top}_\mathbb{C}$. The formula in terms of the Euler characteristic follows from the definitions.} In particular, the tautological perfect complex defines a map
\begin{equation}\label{AJmap}\Fred(\Hc) \xrightarrow{\Index} {}K^{top}_\mathbb{C},\end{equation}
sending $A \in \Fred(\Hc)$ to $\chi(\Hc \xrightarrow{A} \Hc)$.
\begin{theorem}[Atiyah, J\"{a}nich]\label{AJthm}
    The map $\Fred(\Hc) \xrightarrow{\Index} {} K^{top}_{\mathbb{C}}$ is an equivalence.
\end{theorem}
Let $R$ be a ring. The Sato Grassmannian $\Gr(R((t)))$ can be understood as an analogue of the space $\Fred(\Hc)$. We have a map
$$\pi_{R[[t]]}\colon R((t)) \to R[[t]],$$
given by forgetting the principal part of a formal Laurent series. Utilizing this map, we observe that a lattice $L\subset R((t))$ corresponds to the operator
\begin{equation}\label{latcomp}
L \xrightarrow{\pi_{R[[t]]}|_L} R[[t]].
\end{equation}
This operator has finite-dimensional kernel and cokernel, so \eqref{latcomp} allows us to think of lattices as algebraic Fredholm operators. In the Hilbert space setting, this identification of a lattice with an operator defines a weak equivalence between $\Fred(\Hc)$ and the Segal--Wilson analogue of the Sato Grassmannian (c.f. \cite[Chapter 7]{PrS:86}).

We now consider the tautological complex of $R$-modules
\begin{equation*}
    \gamma_\bullet\to\Gr(R((t))),
\end{equation*}
whose fiber at a lattice is the complex \eqref{latcomp}. Just as above, this perfect complex corresponds to a classifying map
\begin{equation*}
    \Gr(R((t)))\to\Perf(R)^\times,
\end{equation*}
where now $\Perf(R)$ is the classifying stack of perfect complexes of $R$-modules. Composing with the map $\Perf(R)^\times\to{}K_R$, we obtain an analogue of \eqref{AJmap}
\begin{equation*}
    \Gr(R((t)))\to{}K_R.
\end{equation*}
However, from the perspective of the Atiyah--J\"{a}nich theorem, it is a crude analogue of \eqref{AJmap}: because the source is a presheaf of sets while the target is a presheaf of spaces, it cannot possibly be an equivalence.

A richer analogue of \eqref{AJmap} exists. The Sato Grassmannian $\Gr(R((t)))$ is a torsor for the group Ind-scheme $\Aut(R((t)))$. In particular, we can view an automorphism $g\in\Aut(R((t)))$ as a Fredholm operator by the assignment
\begin{equation}\label{autcomp}
g \mapsto gR[[t]] \xrightarrow{\pi_{R[[t]]}} R[[t]]
\end{equation}
where $gR[[t]]$ denotes the translate of the lattice $R[[t]]$ under $g$. The analogue of $\Aut(R((t)))$ in the Hilbert space setting is the ``restricted general linear group'' $\GL_{\mathrm{res}}(\Hc,\Hc^+)$ of a polarized Hilbert space.\footnote{The restricted general linear group consists of bounded invertible operators whose commutator with the projection onto $\Hc^+$ is Hilbert--Schmidt. For a longer discussion of this group and the Grassmannian, see \cite[Chapter 6]{PrS:86}.} The Segal--Wilson Grassmannian is a homogeneous space for this group, and the analogue of \eqref{autcomp} induces a weak equivalence between the restricted general linear group and the space of Fredholm operators on $\Hc^+$. This justifies us in viewing $\Aut(R((t)))$ as a richer algebraic analogue of the space of Fredholm operators. Tracing through the discussion above, we obtain a map
\begin{equation}\label{autmap}
    \Aut(R((t)))\to {}K_R,
\end{equation}
sending $g$ to $\chi(gR[[t]] \xrightarrow{\pi_{R[[t]]}}R[[t]])$.
The same reasoning as above shows that this map cannot be an equivalence.

However, as we now explain, the infinite loop space $\Omega K_{\elTate(\C)}$ should be understood as an analogue of $\Fred(\Hc^+)$. A detailed argument follows from the $+$-construction.

Recall that the \emph{perfect radical} of a group $G$ is the largest proper subgroup $P\subset G$ such that $[P,P]=P$. Every group has a perfect radical (c.f. \cite[Remark 4.1.5]{k-book}), and the perfect radical is a normal subgroup.
\begin{definition}
    Let $X$ be a connected space. A \emph{$+$-construction} on $X$ is a map $X\to X^+$ such that the induced map on integral homology is an isomorphism and such that the kernel of the induced map $\pi_1(X)\to\pi_1(X^+)$ is the perfect radical of $\pi_1(X)$.
\end{definition}
A theorem of Quillen (c.f. \cite[Theorem 2.1]{Ger:72}) shows that a $+$-construction exists and is unique up to homotopy equivalence.
\begin{proposition}\label{prop:plus}
    Let $R$ be a ring. The canonical map
    \begin{equation*}
        \Omega (B\Aut(R((t))))^+\to\Omega K_{\elTate(R)}
    \end{equation*}
    is an equivalence.
\end{proposition}
\begin{proof}
    We show that $\Omega (B\Aut(R((t))))^+\simeq\Omega K_{\elTatea(R)}$. By Corollary \ref{cor:tatek=k'}, this will imply the result for uncountable Tate objects.

    For any ring, the category $\elTatea(R)$ of countable elementary Tate modules is split exact \cite[Prop. 5.23]{Braunling:2014fk}. So, its $K$-theory as an exact category is equivalent to the $K$-theory of the symmetric monoidal groupoid $\Tatea(R)^\times$.

    Following Weibel \cite[Theorem 4.4.10]{k-book}, to characterize the $K$-theory of a symmetric monoidal category $(S,\otimes)$ in terms of the $+$-construction, it suffices to show that
    \begin{enumerate}
        \item for any $a,b\in S$, the canonical map $\Aut(a)\to \Aut(a\otimes b)$ is an injection, and
        \item there exists a sequence of objects $\{s_i\}_{i=0}^\infty\subset S$, such that for every $b\in S$, there exists $b'\in S$ such that $b\otimes b'\cong \otimes_{i=0}^n s_i$ for some $n$.
    \end{enumerate}
    Given such a sequence $\{s_i\}$, define $\Aut(S):=\colim_n \Aut_S(\otimes_{i=0}^n s_i)$. Then
    \begin{equation*}
        \Omega(B\Aut(S))^+\cong\Omega K_{S}.
    \end{equation*}
    For $(S,\otimes)=(\elTatea(R),\oplus)$, the first condition is immediately satisfied. For the second, it suffices to observe that every countable elementary Tate module is a direct summand of $R((t))$ (see \cite[Prop. 5.24]{Braunling:2014fk}). Taking $s_0:=R((t))$ and $s_i=0$ for $i>0$, we obtain a sequence of the desired form and conclude the result.
\end{proof}

\begin{lemma}
    The canonical map $B\GL_{\mathrm{res}}(\Hc,\Hc^+)\to (B \GL_{\mathrm{res}}(\Hc,\Hc^+))^+$ is a weak equivalence.
\end{lemma}
\begin{proof}
    This is an immediate consequence of the definition of the $+$-construction and the isomorphism $\pi_1(B \GL_{\mathrm{res}}(\Hc,\Hc^+))\cong\mathbb{Z}$. Recall that this isomorphism arises from the sequence of isomorphisms
    \begin{align*}
        \pi_1(B \GL_{\mathrm{res}}(\Hc,\Hc^+))&\cong\pi_0(\GL_{\mathrm{res}}(\Hc,\Hc^+))\\
        &\cong\pi_0(\Fred(\Hc^+)),
    \end{align*}
    and from the fact that the classical index map $A\mapsto\dim(\ker(A))-\dim(\coker(A))$ induces a bijection between $\pi_0(\Fred(\Hc^+))$ and $\mathbb{Z}$ (c.f. \cite[Theorem 5.35]{Dou:72}).
\end{proof}
Along with the canonical equivalence $\GL_{\mathrm{res}}(\Hc,\Hc^+)\to^\simeq \Omega B \GL_{\mathrm{res}}(\Hc,\Hc^+)$, this gives the following.
\begin{corollary}
    The canonical map $\GL_{\mathrm{res}}(\Hc,\Hc^+)\to\Omega (B \GL_{\mathrm{res}}(\Hc,\Hc^+))^+$ is a weak equivalence.
\end{corollary}
Combined with the equivalence $\GL_{\mathrm{res}}(\Hc,\Hc^+)$ and $\Fred(\Hc^+)$, we see that $\Omega(B\GL_{\mathrm{res}}(\Hc,\Hc^+)^+)$ is another model for $\Fred(\Hc^+)$, and that $\Omega K_{\elTate(R)}$ and, by extension, $\Omega K_{\elTate(\C)}$ is its algebraic analogue.

\subsubsection{The Index Map and Perfect Complexes}
We now give an alternative account of the index map for a split exact, idempotent complete category $\C$. Throughout this section we work with \emph{countably generated} Tate objects. The category $\elTatea(\C)$ is split exact if $\C$ is \cite[Proposition 5.23]{Braunling:2014fk}. Let $V \in \elTatea(\C)$, and $L \in \Gr(V)$ a lattice in $V$. We choose a splitting $\pi\colon V \to L$; as a first step we define a complex $(L' \xrightarrow{\pi|_{L'}} L)$ for every lattice $L' \in \Gr(V)$, analogous to the one of \eqref{latcomp}.

\begin{definition}\label{defi:complexes}
Let $\C$ be a idempotent complete, split exact category. We denote by $\Perf(\C)$ the dg-category of perfect complexes in $\C$, i.e. the Verdier localization of the pre-triangulated dg-category $\Ch^b(\C)$ at the localizing sub-category of acyclic complexes (cf. \cite[Sects. 2,~4]{MR1667558}). The inverse of the equivalence of $K$-theory spaces,\footnote{cf. Gillet--Waldhausen's theorem \cite[Thm. 1.11.7]{MR1106918}.} induced by the exact functor $\C \to \Perf(\C)$ will be denoted by $\chi\colon K_{\Perf(\C)} \xrightarrow{\cong} K_{\C}$.
\end{definition}

\begin{definition}\label{def:complex}
    We denote by $\Gr_{(L,L')}(V)$ the codirected set of lattices $N$ in $V$, which are contained in $L$ and $L'$. It parametrizes a diagram in the category $\Ch^b(\C)$, which sends $N$ to $(L'/N \xrightarrow{\pi|_{L'}} L/N)$ (concentrated in degrees $0$ and $1$ respectively). By virtue of the functor $\Ch^b(\C) \to \Perf(\C)$, we obtain a $\Gr_{(L,L')}(V)$-diagram in $\Perf(\C)$. One sees that all inclusions $N\into N'$ induce equivalent objects in $\Perf(\C)$, since the fiber of the resulting map of cones, is equivalent to
    $$\mathrm{fib}(N'/N \xrightarrow{id} N'/N) \cong 0.$$
    Since $\Gr_{(L,L')}(V)$ is cofiltering, the homotopy limit of this system is canonically equivalent to the complex $(L'/N \to L/N)$ for all $N \in \Gr_{(L,L')}(V)$. We define $(L' \xrightarrow{\pi|_{L'}} L)$ to be
    $$\lim_{\Gr_{(L,L')}(V)} (L'/N \to L/N) \in \Perf(\C).$$
\end{definition}

After having clarified the notation, we are able to give the following interpretation of the index map.

\begin{proposition}\label{prop:index=AJ}
    Let $\C$ be an idempotent complete, split exact category, $V\in\elTatea(\C)$ and let $L\hookrightarrow V$ be a lattice. Fix a splitting $\pi_L\colon V\to L$ of the inclusion $L\hookrightarrow V$. The composition
    \begin{equation*}
        \Aut(V) \to \Omega K_{\elTatea(\C)} \xrightarrow{\Index} {}K_{\C}
    \end{equation*}
    is equivalent to the map
    $$\Aut(V) \to {}K_{\C},$$
    which sends $g$ to $\chi(gL \xrightarrow{\pi_L} L)$ (cf. Definition \ref{defi:complexes}).
\end{proposition}

\begin{proof}
We claim that we have a homotopy bicartesian square
\[
\xymatrix{
(gL \xrightarrow{\pi|_{gL}} L) \ar[r] \ar[d] & gL/N \ar[d] \\
0 \ar[r] & L/N,
}
\]
in $\Perf(\C)$, where $N \in \Gr_{(L,gL)}(\C;V)$ is a lattice contained in both $L$ and $gL$. To see this, we observe that $(gL \rightarrow L)[1]$ is equivalent to
$$\mathrm{cone}(gL/N \to L/N).$$
The cone of a morphism of cochain complexes is a model for the cofiber. In particular, $(gL \rightarrow L) \cong \mathrm{cone}(gL/N \to L/N)[-1]$ is a model for the fiber. We obtain the asserted bicartesian square.

By the definition of $K$-theory, this implies
$$\chi(gL \xrightarrow{\pi|_{gL}} L) \simeq [L/N] - [gL/N].$$
The latter agrees with $\Index(g)$.
\end{proof}

We can summarize the above discussion as follows. For an elementary Tate object $V$ in an idempotent complete, split exact category $\C$, a lattice $L\hookrightarrow V$ gives rise to a canonical homotopy commuting square (i.e. commuting square in the $\infty$-category of spaces)
\begin{equation}\label{AJsquare}
\xymatrix{
\Aut(V) \ar[r] \ar[d] & \Gr(V) \ar[d] \\
\Omega K_{\elTate(\C)} \ar[r] & {}K_{\C},
}
\end{equation}
where the top horizontal map sends $g$ to the lattice $gL$; and the right hand side vertical arrow maps the lattice $L'$ to $\chi(L' \xrightarrow{\pi|_{L'}} L)$.
The lower left, upper right and upper left corners are all algebraic analogues of $\Fred(\Hc)$, while each map to ${}K_\C$ is an algebraic analogue of the map \eqref{AJmap}.

By Theorem \ref{thm:deloop}, the bottom horizontal map, i.e. the index map, is an equivalence
$$\Index\colon\Omega K_{\elTatea(\C)}\xrightarrow{\simeq} K_{\C}.$$
We view this as an \emph{algebraic analogue} of the Atiyah-J\"anich equivalence between the space of Fredholm operators and the classifying space of topological complex $K$-theory (Theorem \ref{AJthm}).

\section{$K$-Theory Torsors}\label{sub:torsor}
Let $R$ be a ring. Denote by $\elTatec(R):=\elTatec(P_f(R))$ the category of elementary Tate $R$-modules.
{ In \cite{Saito:2014fk}, Sho Saito gives a construction of a $K$-theory torsor. He uses the abstract delooping equivalence \[
K_{\elTate(\C)}\to^{\sim } B K_{\C}
\] from \cite{Saito:2012uq}, applied to the category $\C:=\elTatec(R)$. Concretely, he constructs the classifying map for the torsor, and then employs a concept of higher-homotopical torsors in the framework of $\infty$-topoi, as has been developed by \cite{nikolaus2015principal}. Such torsors generalize classical torsors in that the fiber no longer needs to be a (strict) group, but can be a group object in the $\infty$-category of $\mathsf{Spaces}$.

We will look at his construction from a new angle. Instead of the delooping equivalence, we put our index map from Definition \ref{def:CatIndexMap} centerstage. That is, we work with the classifying map as given by Corollary \ref{cor:index}. As in Saito's setup, via \cite{nikolaus2015principal}, this describes a torsor. A priori there is no reason why this torsor should have much to do with Saito's. Moreover, we have a number of tools available for our torsor, e.g., the combinatorial model of Section \ref{sub:indexaut}.

In the context of the ring $R$ our approach amounts to considering our index map}
\begin{equation*}
    K_{\elTatec(R)}\to^{\Index} BK_R
\end{equation*}
as the classifying map of a $K_R$-torsor over $K_{\elTatec(R)}$. We shall explain how classical dimension and determinantal torsors arise as truncations of this $K_R$-torsor. We hope that this will provide the beginning of a satisfactory answer to \cite[Problem 5.5.3]{MR2181808} as well as shed light on the relation between the various torsors arising in \cite{Kapranov:fk}, \cite{MR1988970}, and \cite{MR2656941}. {As we shall see below,} our treatment via the index map yields the dual of the torsor classified by Saito's map, and in a form which can be directly compared to the classical constructions of dimensional and determinantal torsors.

Recall that the category of countable Tate $R$-modules, $\Tatec(R)$, is the idempotent completion of the category of elementary Tate $R$-modules $\elTatec(R)$. The assignments $R\mapsto \Tatec(R)$ and $R\mapsto \elTatec(R)$ are functorial with respect to flat base change, and we view them as defining presheaves of categories on the Nisnevich site of rings.

From Drinfeld, we see that the inclusion $\elTatec(R)\into \Tatec(R)$ is Nisnevich-locally an equivalence \cite[Theorem 3.4]{MR2181808}, and further that $\Tatec(-)$ is a Nisnevich sheaf \cite[Theorem 3.3]{MR2181808}. In other words, $\Tatec(-)$ is the Nisnevich sheafification of $\elTatec(-)$.

The same holds after we pass to $K$-theory. From Thomason--Trobaugh \cite[Theorem 10.8]{MR1106918} and the local vanishing of $K_{-1}$ \cite[Theorem 3.7]{MR2181808}, we know that the (homotopy) Nisnevich sheafification of $BK_R$ is given by $\Omega^{\infty} \Sigma\Kb_R$, where $\Kb_R$ is the non-connective $K$-theory spectrum of $R$. { We remark that Nisnevich-local vanishing of $K_{-1}$ is used to similar effect in Osipov--Zhu's \cite[p. 28]{osipov2015two} and Saito's \cite{Saito:2014fk}.}
Denote the (homotopy) Nisnevich sheafification of a presheaf $F$ of infinite loop spaces by $L(F)$.\footnote{N.b. the $\infty$-category of sheaves of infinite loop spaces is the $\infty$-categorical localization of the $\infty$-category of presheaves of such at the local equivalences. The $L$ stands for ``localization'' functor.} Drinfeld's observations imply that the natural map $K_{\elTatec(R)}\to L(K_{\elTatec(R)})$ extends to a map
\begin{equation*}
    K_{\Tatec(R)}\to L(K_{\elTatec(R)}).
\end{equation*}
Sheafifying the index map, we obtain a natural map
\begin{equation*}
    K_{\Tatec(R)}\to L(K_{\elTatec(R)})\to^{L(\Index)} \Omega^\infty \Sigma\Kb_R.
\end{equation*}
By Theorem \ref{thm:deloop}, this map gives natural isomorphisms on $\pi_i$ for $i>0$. By \cite[Theorem 3.6]{MR2181808}, we also have that it gives an isomorphism on $\pi_0$. We conclude the following {(compare with the role of the Nisnevich topology in Saito's work \cite{Saito:2014fk})}:
\begin{proposition}
    The index map extends to an equivalence of Nisnevich sheaves of infinite loop spaces
    \begin{equation}\label{tateindex}
        K_{\Tatec(-)}\to^{\Index}_\simeq \Omega^\infty\Sigma\Kb_{-}.
    \end{equation}
\end{proposition}
Below, we explain how the 1 and 2-truncations of this map give rise to the dimension and determinantal torsors of \cite{Kapranov:fk}.

We can also consider the category $2\text{-}\elTatec(R)$ of \emph{elementary 2-Tate $R$-modules}, defined by $2\text{-}\elTatec(R):=\elTatec(\Tatec(R))$ (cf. \cite[Section 7]{Braunling:2014fk}). In this setting, the index map takes the form
\begin{equation*}
    K_{2\text{-}\elTatec(R)}\to^{\Index} BK_{\Tatec(R)}.
\end{equation*}
Post-composing with \eqref{tateindex}, we obtain a natural map
\begin{equation}\label{twotateindex}
    K_{2\text{-}\elTatec(R)}\to^{\Index^2} B\Omega^\infty\Sigma\Kb_R.
\end{equation}
We will explain how the 3-truncation of this map gives rise to the 2-gerbe of \cite{MR2656941}. 

\subsection{The Index Torsor for Elementary Tate Objects}
We begin by considering the general case of an idempotent complete exact category $\C$. In Corollary \ref{cor:index} we introduced the map
$$\Index\colon \elTate(\C)^{\grp} \to BK_{\C} \cong |S_\bullet(\C)^{\times}|,$$
by replacing $\elTate(\C)^{\grp}$ by the geometric realization $|\Gr_{\bullet}^{\leq}(\C)^\times|$, and using the natural map of simplicial groupoids
$$\Gr^{\leq}_{\bullet}(\C)^\times \to S_\bullet(\C)^{\times},$$
which sends $(L_0 \into \cdots \into L_n\into V)$ to $(L_1/L_0 \hookrightarrow \cdots \hookrightarrow L_n/L_0)$. Let
\begin{equation*}
    \Tc\to\elTate(\C)^\times
\end{equation*}
be the $K_{\C}$-torsor classified by this map, and let $\Tc|_V$ be the fiber of this torsor at an elementary Tate object $V$. We now explain how the index map allows us to give an elementary description of a part of the space of sections of $\Tc|_V$.

We begin by observing that $V$ determines a commuting square
\begin{equation*}
    \xymatrix{
        \Gr^\le_\bullet(V) \ar[r] \ar[d] & \Gr^\le_\bullet(\C)^\times \ar[d] \\
        \ast \ar[r]^(.4){V} & \elTate(\C)^\times.
    }
\end{equation*}
By Theorem \ref{thm:sato_filtered} (and its corollary, Proposition \ref{prop:gr}), the vertical maps induce equivalences after realization.

We also note that the contractibility of $P^r S_\bullet(\C)^\times$ (Lemma \ref{lemma:pathcontract}) implies that the map
\begin{equation*}
    P^r S_\bullet(\C)^\times\to S_\bullet(\C)^\times
\end{equation*}
becomes equivalent, after realization, to the universal $K_{\C}$-torsor
\begin{equation*}
    \ast\to BK_{\C}.
\end{equation*}

Therefore, from our construction of the index map, we see that every commuting triangle of the form
\begin{equation}\label{Ktorsorssection}
    \xymatrix{
        && P^r S_\bullet(\C)^\times \ar[d] \\
        \Gr_\bullet^\le(V) \ar[r] \ar[urr]^f & \Gr_\bullet^\le(\C)^\times \ar[r]^{\Index} & S_\bullet(\C)^\times
    }
\end{equation}
determines a section of the $K_{\C}$-torsor $\Tc|_V$. Unpacking the definitions, we see that sections of this form admit a very classical description.
\begin{proposition}
    The data of a map $f\colon \Gr_\bullet^\le(V)\to P^r S_\bullet(\C)^\times$, fitting into a triangle of the form \eqref{Ktorsorssection}, consists of:
    \begin{enumerate}
        \item a map $f\colon \Gr(V)\to \ob \C$, and
        \item for each nested sequence of lattices $L_0\into\cdots\into L_n$, a sequence $f(L_0)\into\cdots\into f(L_n)$ of admissible monics in $\C$ such that the assignment of monics to monics is functorial and such that
            \begin{align*}
                f(L_1)/f(L_0)\into\cdots\into f(L_n)/f(L_0)&=\Index(L_0\into\cdots\into L_n)\\
                &=L_1/L_0\into\cdots\into L_n/L_0.
            \end{align*}
    \end{enumerate}
\end{proposition}

From the perspective of $K$-theory, we see that $f$ encodes a map
\begin{align*}
    \Gr(V)&\to K_{\C}\\
    L&\mapsto f(L)
\end{align*}
along with a coherent collection of homotopies
\begin{equation*}
    f(L_0)+L_1/L_0\simeq f(L_1)
\end{equation*}
for every nested pair of lattices $L_0\into L_1$. In particular, the coherence conditions are neatly encoded in the higher simplices of the $S$-construction.

We can also observe, as a consequence of Theorem \ref{thm:sato_filtered}, that such a map $f$ is determined, up to homotopy, by its value on a single lattice. Indeed, for any two lattices $L$ and $L'$, there exists a common enveloping lattice $N$. Any such $N$ determines a homotopy
\begin{equation*}
    f(L')\simeq f(L) + N/L - N/L',
\end{equation*}
and the machinery of Section \ref{sub:indexaut} provides a systematic way of relating the homotopies associated to different choices of $N$.

It is worth noting, that without further work, we cannot extend the above to a description of the entire space of sections of the $K_{\C}$-torsor $\Tc|_V$. However, for several truncations of this torsor of classical interest, we can describe the entire space of sections in these terms. In the following sections, we examine these truncated torsors for $\C=P_f(R)$ in more detail. For a related discussion of these truncations in the abstract setting, see \cite{MR3008236}.\footnote{Note that we replace Previdi's condition that $\C$ satisfies ``$AIC+AIC^{\op}$'' with the condition that $\C$ is idempotent complete. Mutatis mutandis, Previdi's discussion of dimension and determinant torsors now applies to the 1 and 2-truncations of the $K_{\C}$-torsor considered here.}

\subsection{The Dimension Torsor}\label{sub:dimtor}
In this section we relate the 1-truncation of \eqref{tateindex} to the dimension torsor of \cite{Kapranov:fk} (see also \cite{MR2181808}). By descent, it suffices to treat the restriction of \eqref{tateindex} to $K_{\elTatec(R)}$.

We consider an elementary Tate $R$-module $V \in \elTatec(R)$. The rank of finitely-generated projective $R$-modules defines a natural map
$$K_0(R) \to \mathbb{Z}^{\pi_0(\Spec R)}.$$
For ease of notation, we assume from now on that $\Spec R$ is connected. The general case can be recovered by Zariski sheafification.

The Tate $R$-module $V$ gives rise to the $\mathbb{Z}$-torsor, given by the set
\begin{equation*}
    T(V) = \{f\colon \Gr(V) \to \mathbb{Z}|f(L_1) = f(L_0) + \rk(L_1/L_0),\; \forall\;(L_0 \leq L_1) \in \Gr^{\leq}_1(V)\}.
\end{equation*}
Such $f$ are called \emph{dimension theories} in \cite{Kapranov:fk}. We see that $T(V)$ is a $\mathbb{Z}$-torsor as follows.  First, the action of $k \in \mathbb{Z}$ on $T(V)$ is defined pointwise, by $f \mapsto f + k$. Second, one uses that any two lattices $L_0$ and $L_1$ admit a common upper bound to show that a function $f\in T(V)$ is determined by its value at a single lattice $f(L)$.

Since $\Gr(-)\colon \elTatec(R)^{\grp} \to \Set$ is a functor, we see that the construction $T(V)$ is functorial as well:
$$T(-)\colon \elTatec(R)^{\grp} \to B\mathbb{Z}.$$
Here $B\mathbb{Z}$ denotes the groupoid of $\mathbb{Z}$-torsors.

For every elementary Tate object $V$ we have a map $B\Aut(V) \to \elTate(\C)^{\grp}$. We conclude therefore the existence of a $\mathbb{Z}$-torsor on $B\Aut(V)$, classified by the map
$$B\Aut(V) \to B\mathbb{Z}.$$

\begin{proposition}\label{prop:dimtor}
The map $\elTatec(R)^{\grp} \to B\mathbb{Z}$, obtained as the composition of the Index map $\elTatec(R)^{\grp} \xrightarrow{\Index} BK_{R}$ with the map
$BK_{R} \xrightarrow{B(\rk)} B\mathbb{Z}$ (induced by the rank of finitely generated projective modules), is equivalent to the map $T(-)$ classifiying the \emph{dimensional torsor}.
\end{proposition}
\begin{proof}
    The map $B(\rk)\colon BK_{R} \to B\mathbb{Z}$ is equivalent to the geometric realization of the map
    $$S_\bullet(\C)^{\times} \to B_{\bullet}\mathbb{Z},$$
    which sends $(0 \hookrightarrow X_1 \hookrightarrow \cdots \hookrightarrow X_k)$ to $(\rk(X_1), \rk(X_2/X_1),\dots,\rk(X_k/X_{k-1}))$. Consequently, $B(\rk) \circ \Index$ is equivalent to the geometric realization of
    $$\Gr^{\leq}_{\bullet}(P_f(R))^\times \to B_{\bullet}\mathbb{Z},$$
    which sends $(V;L_0 \hookrightarrow \cdots \hookrightarrow L_k)$ to $(\rk(L_1/L_0),\dots,\rk(L_k/L_{k-1}))$. This map induces an augmentation of $\Gr^{\leq}_{\bullet}(P_f(R))^\times$ by $B\mathbb{Z}$, hence a $\mathbb{Z}$-torsor on the simplicial space $\Gr^{\leq}_{\bullet}(P_f(R))^\times$.

    Since $B_0\mathbb{Z} = \{\star\}$ is a singleton, the $\mathbb{Z}$-torsor above trivializes when pulled back to $\Gr^{\leq}_{0}(P_f(R))^\times$. Hence, every choice of a lattice $L \subset V$ induces a trivialization. Given a nested pair of lattices $L \leq L'$, the two corresponding trivializations differ precisely by $\rk(L'/L)$. We therefore obtain for the space of sections of the torsor $B(\rk) \circ \Index$ over a connected component $B\Aut(V) \subset \elTatec(R)^{\grp}$ the set of functions $f\colon \Gr(P_f(R),V) \to \mathbb{Z}$, satisfying $f(L') = f(L) + \rk(L'/L)$ for all nested pairs $L \leq L'$ of lattices in $V$. This is precisely the definition of the torsor $T(V)$.
\end{proof}

\subsection{The Determinant Torsor}\label{sub:dettors}
We now investigate the 2-truncation of \eqref{tateindex} and relate it to the graded determinant torsor of \cite{Kapranov:fk} (see also \cite{MR1988970} and \cite{MR2181808}).

Denote by $\PicZR$ the symmetric monoidal groupoid of graded lines $\PicZR$. A graded line is a pair $(L,n)$, where $L$ is an invertible $R$-module, and $n\colon \Spec R \to \mathbb{Z}$ a Zariski-locally constant function. The usual symmetry constraint of tensoring $R$-modules
\begin{equation*}
    \phi_{L,M}\colon L \otimes M \simeq M \otimes L
\end{equation*}
will be modified by a sign:
\begin{equation*}
    (L,n) \otimes (M,m) \simeq (L \otimes M, n + m) \xrightarrow{(-1)^{mn}\phi_{L,M}} (M \otimes L, m + n) \simeq (M,m) \otimes (L,n).
\end{equation*}
For $M\in P_f(R)$, the assignment $M\mapsto(\Lambda^{top}M,\rk(M))$ extends to a map of infinite loop spaces\footnote{The sign in front of $\phi_{L,M}$ ensures that this map is an infinite loop map.}
\begin{equation}\label{eqn:gr_det}
    \detZ\colon K_R\to \PicZR.
\end{equation}
We have a natural morphism
\begin{equation*}
|\elTatec(R)^\times| \to K_{\elTatec(R)} \xrightarrow{\Index} B K_{R}.
\end{equation*}
By composition with the graded determinant of equation \eqref{eqn:gr_det}, we obtain a morphism
\begin{equation*}
B(\detZ) \circ \Index \colon |\elTatec(R)^\times| \to B\PicZR.
\end{equation*}
The target is a sheaf of groupoids, so this extends to a morphism
\begin{equation}\label{detseq}
    B(\detZ)\circ\Index\colon|\Tatec(R)^\times|\to B\PicZR.
\end{equation}
In particular, we see that $\Tatec(R)^\times$ is endowed with a canonical $\PicZR$-torsor.

\begin{definition}
    Define the \emph{determinant torsor} $\Dc_R\to \Tatec(R)^\times$ to be the $\PicZR$-torsor classified by the map of \eqref{detseq}.
 \end{definition}

\begin{proposition}\label{prop:dettor}
    Let $R$ be a ring, and let $V$ be an elementary Tate $R$-module. The space of sections $\Gamma(\Spec(R),\Dc_R|_{\{V\}})$ of the restriction of the determinant torsor to $\{V\}\in\Tatec(R)^\times$ is equivalent to the space of maps $\Delta\colon\Gr(V)\to\PicZR$ equipped with a coherent collection of equivalences
    \begin{equation}\label{eqn:det1}
        \Delta(L')\simeq\Delta(L)\otimes\detZ(L'/L)
    \end{equation}
    for every nested pair of lattices $L\into L'$. The coherence condition amounts to the commutativity of the diagram
    \begin{equation}\label{eqn:det2}
    \xymatrix{
    \Delta(L'') \ar[r] \ar[d] & \Delta(L') \otimes \detZ(L''/L') \ar[d] \\
    \Delta(L) \otimes \detZ(L''/L) \ar[r] & \Delta(L) \otimes \detZ(L'/L) \otimes \detZ(L''/L'),
    }
    \end{equation}
    for all $(L \leq L' \leq L'') \in \Gr_2(V)$.
    In particular, $\Dc_R|_{\{V\}}$ is equivalent to the determinant torsor of $V$ as described in \cite[Section 5.2]{MR2181808}.
\end{proposition}

\begin{proof}
    The proof is analogous to Proposition \ref{prop:dimtor}, we will therefore only sketch the general argument. The composition $B(\detZ) \circ \Index$ is equivalent to the geometric realization of the map
    $$\Gr^{\leq}_{\bullet}(P_f(R))^\times \to B_{\bullet}(\PicZR),$$
    sending $L_0\into\cdots L_n\into V$ to $(\detZ(L_1/L_0),\dots,\detZ(L_n/L_{n-1}))$. Using the canonical augmentation of $B_{\bullet}(\PicZR)$, we obtain an augmentation of $\Gr^{\leq}_{\bullet}(P_f(R))^\times \to B\PicZR$, hence a $\PicZR$-torsor on the simplicial space $\Gr^{\leq}_{\bullet}(P_f(R))^\times$. Since $\Gr^{\leq}_{0}(P_f(R))^\times$ factors through $B_0(\PicZR) = \{\star\}$, the corresponding $\PicZR$-torsor is trivialized on the cover $\Gr_0^{\le}(P_f(R))^\times \to \elTatec(R)^{\grp}$. To determine the space of sections, it suffices to describe the descent conditions for the space of sections of the trivial torsor on $\Gr_0^\le(P_f(R))^\times$. By inspection, we see that the trivializations corresponding to a nested pair of lattices $L \into L'\into V$ differ by $\detZ(L'/L)$, as claimed in \eqref{eqn:det1}. The coherence condition \eqref{eqn:det2} follows as well, but we need to show that no further condition has to be imposed. To see this we observe the following: since the space of sections is a groupoid, it is determined by restricting the simplicial object to the $2$-skeleton. Hence, no further coherence conditions appear.
\end{proof}

\begin{rmk}
    In the previous section, we observed that the rank of finitely generated projective modules defines a map
    $B\Aut(V)\to B\mathbb{Z}$ for any countable Tate module $V$. Applying $\Omega(-)$, we obtain a group homomorphism
    $\nu\colon\Aut(V) \to \mathbb{Z}.$ Similarly, for each countable Tate $R$-module $V$, the map
    $B\Aut(V) \to B\PicZR$ induces a monoidal map $\Aut(V) \to \PicZR.$ These maps are in correspondence with \emph{graded central extensions} of $\Aut(V)$. In particular, we see that the map $\nu$ extends to a graded central extension of $\Aut(V)$.
\end{rmk}

\subsection{The Determinant 2-Gerbe}\label{sec:2ger}
Let $V$ be an elementary $2$-Tate module. We now explain how one can recover Arkhipov--Kremnitzer's 2-gerbe of \emph{gerbal theories} (cf. \cite{MR2656941}) as a truncation of \eqref{twotateindex}.

Let $R$ be a ring, and let $V$ be an elementary $2$-Tate module. The map
\begin{equation*}
    B\Aut(V)\to 2\text{-}\elTate(R)^\times\to^{\Index^2} B\Omega^\infty\Sigma\Kb_R\to^{B^2\detZ} B^2\PicZR
\end{equation*}
classifies a 2-gerbe on $B\Aut(V)$. By construction, this map arises as the geometric realization of the simplicial map
\begin{equation*}
    \Gr_{\bullet}^{\leq}(V)^\times \to S_{\bullet}\Tate(R)^\times \to B_{\bullet}B\PicZR,
\end{equation*}
where the second map sends $(L_0 \hookrightarrow \cdots \hookrightarrow L_n\hookrightarrow V)$ to $(\Dc_R|_{L_1/L_0},\dots,\Dc_R|_{L_n/L_{n-1}})$. The space of sections of this 2-gerbe is equivalent to the space of maps $\Delta\colon\Gr(V)\to B\PicZR$ equipped with a coherent collection of equivalences
\begin{equation*}
    \Delta(L')\simeq\Delta(L)\otimes\Dc_R|_{L'/L}
\end{equation*}
for a nested pair of lattices $L\into L'$. In particular, when $R$ is a field, $\Gg_R|_{\{V\}}$ is equivalent to the graded determinant 2-gerbe of the $2$-Tate vector space $V$ introduced by Arkhipov and Kremnitzer \cite[Thm. 5]{MR2656941} (see also \cite[Section 5.2]{MR2181808}). We refrain from spelling out the precise coherence conditions, and refer the reader to \emph{loc. cit.} for more details.

Arkhipov and Kremnitzer used the determinant 2-gerbe of a $2$-Tate vector space to construct a higher central extension of $\Aut(V)$. As remarked in \emph{loc. cit.}, because their formalism only allowed them to consider $2$-Tate vector spaces, they were forced to view $\Aut(V)$ as a \emph{discrete group}.  The above construction allows us to define the central extension of the \emph{group scheme} $\Aut(V)$. Namely, the extension is the algebraic 2-group given by the central term in the fiber sequence (in the category of stacks over $\Spec(R)$).
\begin{equation*}
    \PicZR\to\Omega\Gg_R|_{B\Aut(V)}\to\Aut(V).
\end{equation*}
We denote this algebraic 2-group below by $\widehat{\Aut(V)}$. This generalizes the constructions of \cite[Sect. 5.2]{osipov2015two} to automorphism groups of arbitrary elementary $2$-Tate objects.

\subsection{The Determinant Gerbe on the Grassmannian of a $2$-Tate Module}\label{sec:FZ}
In \cite{MR2972546}, Frenkel and Zhu describe a conjectural \emph{determinant gerbe} on the Sato Grassmannian of a $2$-Tate $R$-module in connection with a geometric representation theory of double loop groups. We give a rigorous construction of this gerbe in this section, and use it to construct a \emph{basic categorical representation} of the aforementioned algebraic 2-group $\widehat{\Aut(V)}$.

Let $V$ be an elementary $2$-Tate module. In Section \ref{subsub:fredholm}, we saw how a choice of lattice $L\in\Gr(V)$ gives rise to the square \eqref{AJsquare}
\begin{equation*}
\xymatrix{
\Aut(V) \ar[r] \ar[d] & \Gr(V)\ar[d] \\
\Omega K_{\twoTate(R)}\ar[r] & K_{\Tate(R)}
}
\end{equation*}
Combining this with the index map \eqref{tateindex} and the determinant, we obtain a map
\begin{equation}\label{detgerbe}
    \Gr(V)\to K_{\Tate(R)}\to B\PicZR.
\end{equation}
The construction guarantees that, up to equivalence, this map is independent of the choice of lattice $L$.
\begin{definition}
    Let $V$ be an elementary $2$-Tate $R$-module. The \emph{determinant gerbe} on the Sato Grassmannian $\Gr(V)$ is the $\PicZR$-torsor $\Gg_V\to\Gr(V)$ classified by the map \eqref{detgerbe}.
\end{definition}

\begin{rmk}
    In the present approach, the construction of the determinant gerbe follows from the results of \cite{Braunling:2014fk} on $2$-Tate modules and from the construction of the index map.
\end{rmk}

We can give a more explicit description of the determinant gerbe. Fix a lattice $L\in\Gr(V)$. From the construction, we see that the fiber of the determinant gerbe at a lattice $L'$ is canonically equivalent to
\begin{equation*}
    \Gg_V|_{\{L'\}}\simeq\Dc_R^\vee|_{\{L/N\}}\otimes_{\PicZR}\Dc_R|_{\{L'/N\}}
\end{equation*}
where $N$ is any common sub-lattice of $L$ and $L'$, $\Dc_R$ is the determinant torsor on $\Tate(R)^\times$, $\Dc_R^\vee$ is its dual, and the tensor product of $\PicZR$-torsors is as described in \cite[Section 2.2]{MR2562452}. Equivalently,
\begin{equation*}
    \Gg_V|_{\{L'\}}\simeq\Hom_{\PicZR}(\Dc_R|_{\{L/N\}},\Dc_R|_{\{L'/N\}}).
\end{equation*}
Given a Tate $R$-module $M$, Drinfeld \cite[Section 5.4]{MR2181808}, following Beilinson--Bloch--Esnault \cite{MR1988970}, describes the determinant torsor $\Dc_R|_{\{M\}}$ in terms of invertible modules for a Clifford algebra $\Cl_M$. In more detail, the duality pairing gives a symmetric bilinear form on the Tate module $M\oplus M^\vee$. Denote by $\Cl_M$ the $\mathbb{Z}$-graded Clifford algebra associated to this form, with the grading given by placing $M$ in degree 1, and $M^\vee$ in degree $-1$. We similarly consider graded modules for $\Cl_M$.
\begin{definition}
    Let $M$ be a Tate $R$-module, and let $\Cl_M$ be the graded Clifford algebra described above. A \emph{graded Clifford module} $F$ is a $\mathbb{Z}/2$-graded $R$-module, with a graded action of $\Cl_M$, which is continuous with the respect to the discrete topology on $F$. We say that $F$ is a \emph{graded Fermion module} if the functor $(-)\otimes_R F$, from $\mathbb{Z}/2$-graded $R$-modules to graded $\Cl_M$-modules, is an equivalence of categories.
\end{definition}
Denote by $\FermZ(M)$ the groupoid of graded Fermion modules for $Cl_M$. The groupoid $\FermZ(M)$ carries a natural action of $\PicZR$ given by tensoring a graded Fermion module with a graded line.

\begin{proposition}(Drinfeld \cite[Section 5.4]{MR2181808} following Beilinson--Bloch--Esnault \cite{MR1988970})
    Let $M$ be a Tate $R$-module. The groupoid $\FermZ(M)$ is a $\PicZR$-torsor. For a graded Fermion module $F$, and a lattice $L\subset M$, denote the annihilator of $L\oplus L^\vee\subset M\oplus M^\vee$ in $F$ by $\Delta_F(L)$. Then $\Delta_F(L)\subset F$ is a graded line, with the grading inherited from that of $F$. The assignment $L\mapsto\Delta_F(L)\in\PicZR$ defines a graded determinantal theory, and the assignment $F\mapsto\Delta_F$ induces an equivalence $\FermZ(M)\to^\simeq\Dc_R|_{\{M\}}$ of $\PicZR$-torsors.
\end{proposition}

Now let $V$ be a $2$-Tate $R$-module, and $L\subset V$ a lattice, as above. The proposition allows us to reformulate the fiber of the determinant gerbe $\Gg_V$ at a lattice $L'$ as
\begin{equation*}
    \Gg_V|_{\{L'\}}\simeq\Hom_{\PicZR}(\FermZ_{L/N},\FermZ_{L'/N}),
\end{equation*}
where $N$ is any common sub-lattice of $L$ and $L'$. From the perspective of categories of modules, $\FermZ_M$ is the groupoid of indecomposables in the category $\Mod_{\Cl_M}^{ss}$ of semi-simple $\Cl_M$-modules. A homomorphism of $\PicZR$-torsors is necessarily an equivalence, so a homomorphism $\FermZ_{L/N}\to\FermZ_{L'/N}$ corresponds to an equivalence of $R$-linear categories
\begin{equation*}
    \Mod_{\Cl_{L/N}}^{ss}\to^\simeq\Mod_{\Cl_{L'/N}}^{ss}.
\end{equation*}
According to Morita theory, the groupoid of such equivalences is equivalent to the category of invertible $\Cl_{L/N}-\Cl_{L'/N}$-bimodules.

Let $V$ be a $2$-Tate $R$-module, and let $\Gg_V\to\Gr(V)$ be the determinant gerbe on $\Gr(V)$. Fix a lattice $L\subset V$, the fiber $\Gg_V|_{\{L'\}}$ at any lattice $L'\subset V$ is equivalent to the $\PicZR$-torsor of invertible $\Cl_{L/N}-\Cl_{L'/N}$-bimodules, where $N$ is any common sub-lattice of $L$ and $L'$.

As described in Section \ref{sec:2ger}, the homomorphism $\Aut(V)\to B\PicZR$ gives rise to a central extension $\widehat{\Aut(V)}$ of the group scheme $\Aut(V)$ by the stack $\PicZR$. The constructions guarantee that $\widehat{\Aut(V)}$ acts on the determinant gerbe $\Gg_V$ and that this action lifts the action of $\Aut(V)$ on $\Gr(V)$.

Denote by $P_f^{\mathbb{Z}}(R)$ the category of finitely generated $\mathbb{Z}$-graded projective $R$-modules. The category $P_f^{\mathbb{Z}}(R)$ carries a natural action of $\PicZR$ and the homotopy quotient
\begin{equation*}
    P_f^{\mathbb{Z}}(R)//\PicZR\to B\PicZR
\end{equation*}
defines a bundle of exact categories. The pullback of this bundle along the map $\Gr(V)\to B\PicZR$ defines a bundle of exact categories $\Pc_V\to\Gr(V)$. As with $\Gg_V$, the construction ensures that $\widehat{\Aut{V}}$ acts on $\Pc_V$ and that this action lifts the action of $\Aut(V)$ on $\Gr(V)$. In particular, we obtain an $R$-linear representation of $\widehat{\Aut(V)}$ on the category of global sections of $\Pc_V\to\Gr(V)$.

\begin{definition}
    Let $V$, $\widehat{\Aut(V)}$ and $\Pc_V$ be as above. The \emph{basic representation} of $\widehat{\Aut(V)}$ is given by the action $\widehat{\Aut(V)}\circlearrowleft\Gamma(\Gr(V),\Pc_V)$.
\end{definition}

This picture suggests a more concrete description of the fibers of $\Pc_V\to\Gr(V)$. Let $V$ be a $2$-Tate $R$-module, and let $\Pc_V\to\Gr(V)$ be the bundle of exact categories associated to $\Gg_V$ as above. Given a lattice $L\subset V$, it seems plausible that the fiber of $\Pc_V|_{\{L'\}}$ at any lattice $L'\subset V$ is equivalent to the category of semi-simple $\Cl_{L/N}-\Cl_{L'/N}$-bimodules.

Frenkel and Zhu \cite{MR2972546} have previously constructed a \emph{basic representation} of the $\mathbb{C}$-points of $\widehat{\Aut(V)}$, when $V$ is the $2$-Tate space $\mathbb{C}((s))((t))$ over $\mathbb{C}$ (they denote the $\mathbb{C}$-points of $\widehat{\Aut(V)}$ by $\mathbb{GL}(V)$). In their formulation, the basic representation consists of an action of $\widehat{\Aut(V)}$ on a category of semi-simple modules for a Clifford algebra associated to a lattice of $V$. We expect the following: let $V$ be an elementary $2$-Tate space over $\mathbb{C}$. The category of $\mathbb{C}$-points of the representation $\Gamma(\Gr(V),\Pc_V)$ of the algebraic 2-group $\widehat{\Aut(V)}$ is equivalent to the basic representation of the $\mathbb{C}$-points of $\widehat{\Aut(V)}$ constructed in \cite{MR2972546}.

\bibliographystyle{amsalpha}
\bibliography{master}
\end{document}